\documentclass[11pt,leqno]{article}
\usepackage{amsmath, amscd, amsthm, amssymb, graphics, xypic, mathrsfs, setspace, fancyhdr, times, bm, enumitem}
\usepackage[letterpaper,top=1.05in, bottom=1.05in, left=1.05in, right=1.05in]{geometry}
\usepackage[colorlinks=true,pagebackref=true]{hyperref} 
\hypersetup{backref}

\newcommand{\Spec}{\operatorname{Spec}}

\renewcommand{\hom}{\operatorname{Hom}}

\newcommand{\cplx}{{\mathbb C}}

\newcommand{\Z}{{\mathbb Z}}

\newcommand{\A}{{\mathbb A}}
\newcommand{\Proj}{{\mathbb P}}
\newcommand{\Jo}{{\mathbb J}}
\newcommand{\aone}{{\mathbb A}^1}
\newcommand{\pone}{{\mathbb P}^1}

\newcommand{\gm}{{{\mathbb G}_{m}}}

\newcommand{\MW}{\mathrm{MW}}

\newcommand{\bpi}{\bm{\pi}}

\newcommand{\piaone}{{\bpi}^{\aone}}

\newcommand{\KMW}[1]{{{\mathbf K}_{#1}^{\MW}}}
\newcommand{\KM}[1]{{{\mathbf K}_{#1}^{\mathrm{M}}}}

\newcommand{\Addresses}{{
 \bigskip
 \footnotesize

 J.~Fasel, Institut Fourier, UMR 5582, Université Grenoble Alpes, CS 40700, 38058 Grenoble cedex 9, France \textit{E-mail address:} \url{jean.fasel@univ-grenoble-alpes.fr}

 \medskip

 W.P.~Hornslien, Institut Fourier, UMR 5582, Université Grenoble Alpes, CS 40700, 38058 Grenoble cedex 9, France \textit{E-mail address:} \url{w.hornslien.math@proton.me}
}}

\newcounter{intro}
\theoremstyle{plain}
\newtheorem{thm}{Theorem}[subsection]

\newtheorem{lem}[thm]{Lemma}

\newtheorem{prop}[thm]{Proposition}
\newtheorem*{claim*}{Claim}  

\newtheorem*{thm*}{Theorem}
\newtheorem*{problem*}{Problem}

\newtheorem{thmintro}{Theorem}

\theoremstyle{definition}

\theoremstyle{remark}
\newtheorem{rem}[thm]{Remark}

\numberwithin{equation}{section}

\begin{document}
\pagestyle{fancy}
\renewcommand{\sectionmark}[1]{\markright{\thesection\ #1}}
\fancyhead{}
\fancyhead[LO,R]{\bfseries\footnotesize\thepage}
\fancyhead[LE]{\bfseries\footnotesize\rightmark}
\fancyhead[RO]{\bfseries\footnotesize\rightmark}
\chead[]{}
\cfoot[]{}
\setlength{\headheight}{1cm}

\author{ Jean Fasel\thanks{Jean Fasel was partially supported by the project IRGA2024 –MOTOR – G7H-IRG24E97.}\and William Hornslien\thanks{William Hornslien was partially supported by the project IRGA2024 –MOTOR – G7H-IRG24E97.} }

\title{{\bf Exotic Hopf maps, weight shifting and applications to vector bundles}}
\date{}
\maketitle

\begin{abstract}
Using motivic homotopy theory we produce several explicit polynomial representatives of the suspension of the Hopf map defined over $\Z$. We derive from this computation an explicit rank $2$ vector bundle on the Jouanolou device of $\mathbb{P}^3_\Z$.
\end{abstract}

\begin{footnotesize}
\setcounter{tocdepth}{1}
\tableofcontents
\end{footnotesize}

\section{Introduction}

This paper can be viewed as an explicit version of some results in \cite{Asok20b}, where the problem of producing complex polynomial representatives of homotopy groups of spheres was considered. Recall that P. Baum proved in \cite{Baum67} that any element of the \emph{stable} homotopy groups of spheres can be represented by an admissible quadratic map, i.e., a real polynomial map of degree at most $2$. However, Wood later observed \cite{Wood68} that there are no non-constant polynomial maps $S^n\to S^m$ for $n\geq 2m$, showing that Baum's result fails in the unstable range. This led to the study of complex polynomial representatives, using the fact that the complex algebraic sphere $S_{\mathbb{C}}^n$, viewed as a manifold, is homotopy equivalent to the real sphere $S^n$. At present, no element of a homotopy group of spheres is known that cannot be represented by such a complex polynomial map.

In \cite[Proposition 2.2.1]{Asok20b}, it was shown that the suspension of the Hopf map $S^4\to S^3$ and the square of the Hopf map $S^5\to S^3$ admit complex polynomial representatives. The method relied on motivic homotopy theory and a weight-shifting procedure that we briefly recall. If $k$ is a field and $n\geq 1$, consider the smooth quadrics $Q_{2n-1}:=\{\sum x_iy_i=1\}\subset \A^{2n}_k$ (with coordinates $x_1,\ldots,x_n,y_1,\ldots,y_n$) and $Q_{2n}:=\{\sum x_iy_i=z(1-z)\}\subset \A^{2n+1}_k$ (with coordinates $x_1,\ldots,x_n,y_1,\ldots,y_n,z$). It is well known (e.g. \cite{Asok14}) that these quadrics model motivic spheres in the (unstable) motivic homotopy category $\mathcal{H}(k)$ of Morel–Voevodsky \cite{Morel99}. Moreover, for any smooth affine $k$-scheme $X$, the map $\mathrm{Hom}_{\mathrm{Sch}_k}(X,Q_i)\to [X,Q_i]_{\A^1}:=\mathrm{Hom}_{\mathcal{H}(k)}(X,Q_i)$ is surjective for $i\geq 1$ (\cite[Theorem 4.2.1]{Asok15b}, \cite[Corollary 3.1.1]{Asok22}).

If $k\subset \cplx$, the complex realization functor \cite[\S 3.3]{Morel99}
\[
r_{\cplx}\colon \mathcal{H}(\cplx)\to \mathcal{H}
\]
(left Kan extended from $X\mapsto X(\cplx)$ endowed with the complex topology) sends $Q_i$ to $S^i$. Hence, we obtain a composite map
\[
\mathrm{Hom}_{\mathrm{Sch}_k}(Q_j,Q_i)\to [Q_j,Q_i]_{\A^1}\to \pi_j(S^i)
\]
for any $i,j\geq 1$. Thus, to show that $\alpha\in \pi_j(S^i)$ admits a polynomial representative, it suffices to produce a lift $\alpha'\in [Q_j,Q_i]_{\A^1}$. Since $Q_j$ is a motivic sphere, we have identifications with bi-graded homotopy groups
\[
[Q_j,Q_i]_{\A^1}=\begin{cases} \piaone_{n,n}(Q_i)(k) & \text{ if $j=2n$, }  \\ \piaone_{n-1,n}(Q_i)(k) & \text{ if $j=2n-1$. }\end{cases}
\]
These groups can sometimes be computed using fiber sequences analogous to those in topology. For example, setting $i=3,j=4$ yields
\[
\piaone_{2,2}(Q_3)=\piaone_{2,2}(\mathrm{SL}_2)=\piaone_{2,2}(\A^2\smallsetminus\{0\}),
\]
which fits in an exact sequence (\cite[Theorem 3.3, Proposition 4.4]{Asok12a})
\[
\KM 2(k)/12\times_{\KM 2(k)/2}\mathbf{I}^2(k)\to \piaone_{2,2}(Q_3)(k)\to \mathbf{GW}_1^0(k)\to 0.
\]
If $k$ has characteristic $\neq 2$, then $\mathbf{GW}_1^0(k)=\mathrm{KO}_1(k)\simeq \Z/2\times k^\times/(k^{\times})^2$. When $k=\cplx$, we obtain in particular $\piaone_{2,2}(Q_3)(\cplx)\simeq \Z/2$, and the generator maps to the generator of $\pi_4(S^3)=\Z/2$, namely the suspension of the Hopf map.

In \cite{Asok20b}, a different approach called \emph{weight shifting} was used to produce a motivic lift of the suspension of the Hopf map. We observe that $\piaone_{1,3}(Q_3)=\mathbf{W}$ (the Witt sheaf) and that its generator, as a $\KMW 0$-module, is $2_{\epsilon}$-torsion (with $2_{\epsilon}=\langle 1,-1\rangle\in \mathrm{GW}(k)$). It is shown in \cite[Theorem 2.1.11]{Asok20b} that choosing $-1$ as a primitive $2$nd root of unity induces a morphism of sheaves
\[
\mathbf{W}={}_{2_\epsilon}\piaone_{1,3}(Q_3)\to \piaone_{2,2}(Q_3)/[-1]\piaone_{2,3}(Q_3).
\]
By construction, the image of the generator realizes to the suspension of the Hopf map. Although conceptually natural, this construction is difficult to make explicit and a geometric realization was missing.

In this article we give two explicit constructions of morphisms of quadrics $Q_4\to Q_3$, called \emph{exotic Hopf maps}, whose complex realization is the suspension of the Hopf map. The first construction uses the computation of $\piaone_{2,2}(Q_3)$ above, specifically the morphism $\piaone_{2,2}(Q_3)\to \mathbf{GW}_1^0(k)$ and the summand $\Z/2\subset \mathrm{GW}_1^0(k)$. We first note that $\mathrm{GW}_1^0(k)=\piaone_{2,2}(\mathrm{Sp})$, where $\mathrm{Sp}$ denotes the infinite symplectic group defined over $\Z$. Hence
\[
\piaone_{2,2}(\mathrm{Sp})=[Q_4,\mathrm{Sp}]_{\aone}=\mathrm{GW}_1^2(Q_4)=\mathrm{KSp}_1(Q_4).
\]
Using the edge map in the coniveau spectral sequence for Hermitian $K$-theory, we construct an explicit symplectic matrix $M\in \mathrm{Sp}_8(Q_4)$ representing the generator of $\Z/2$. Connectivity arguments show that this matrix can be reduced by elementary symplectic operations, although carrying this out explicitly is somewhat delicate. Nevertheless, we obtain the following theorem.

\begin{thmintro}
The exotic Hopf map $Q_4\to Q_3$ defined (over $\Z$) by the matrix $\begin{pmatrix}
a & b \\ c & d
\end{pmatrix}$ with

\begin{align*}
a =& 16x_2^3y_1^2y_2+16x_2^2y_1^2z-16x_2^2y_1y_2z-16x_2^2y_1^2-8x_2^2y_1y_2-16x_2^2y_2^2-24x_2y_1z^2 \\ &+20x_2y_1z-8x_2y_2z+8z^3+4x_2y_1-8x_1y_2+12x_2y_2-8z^2-2z+1, \\
b =& -32x_2^3y_2z-32x_1x_2^2y_2+32x_2^3y_2-32x_2^2z^2-48x_1x_2z+64x_2^2z-16x_1^2+40x_1x_2-32x_2^2, \\
c =& -2x_2y_1^3-4x_2y_1^2y_2+2y_1^2z+4y_1y_2z+y_1^2+2y_1y_2+4y_2^2, \\
d =& 4x_2y_1z+8x_2y_2z-4x_2y_1+8x_1y_2-12x_2y_2-4z^2+2z+1.
\end{align*}
complex realizes to the suspension of the Hopf map.
\end{thmintro}

The second construction uses the weight-shifting procedure. Starting from the universal rank $2$ bundle on $Q_5\simeq \A^3\smallsetminus{0}$ classified by a map $Q_5\to \mathrm{BSL}_2$, adjunction yields a map $S^{1,3}\to \mathrm{SL}_2=Q_3$. Since $S^{1,3}$ has no convenient model as a smooth scheme, we instead use a presentation as a quotient of $Q_3\times \gm{}$ (as Nisnevich sheaves of sets). This leads to a morphism of schemes
\[
T\colon Q_3\times \gm{}\to \mathrm{SL}_2=Q_3
\]
which is trivial on $Q_3\vee \gm{}$. Considering a model of the multiplication-by-$2$ map $m'\colon Q_3\to Q_3$, one checks that the composite
\[
Q_3\times \gm{}\xrightarrow{m'\times 1}Q_3\times \gm{}\to Q_3
\]
is homotopically trivial. This produces a morphism
\[
\phi\colon Q_3\times \gm{}\times \A^1\to Q_3
\]
with $\phi(0)=\mathrm{Id}$ and $\phi(1)=T\circ m'$. Evaluating at the point $-1\in\gm{}$ gives a morphism $G\colon Q_3\times\A^1\to Q_3$ satisfying $G(0)=G(1)=\mathrm{Id}$ and then a map $(Q_3\times \A^1)/(Q_3\times S^0)\to Q_3$. The source formally resembles a suspension of $Q_3$, hence $Q_4$. After suitable modifications and using a morphism $Q_3\times(\A^1\smallsetminus\{0,1\})\to Q_4$ constructed in Section \ref{sec:basicsquadrics}, we obtain the desired morphism $Q_4\to Q_3$. The resulting map has relatively high degree, as illustrated by the following theorem.

\begin{thmintro}
The exotic Hopf map $Q_4\to Q_3$ defined (over $\Z$) by the unimodular row $(a',b')$ with
\begin{align*}
a' =& 144x_2^2z^9+144x_2y_1z^9+72x_1x_2z^8-384x_2^2z^8-240x_2y_1z^8-72x_2 y_2z^8-72z^{10}-192x_1x_2z^7\\
&+112x_2^2z^7-128x_2y_1z^7+120x_2y_2z^7+192 z^9+92x_1x_2z^6+384x_2^2z^6+256x_2y_1z^6+28x_2y_2z^6\\
&-92z^8+96x_1x_2z ^5-208x_2^2z^5+48x_2y_1z^5-24y_1^2z^5-68x_2y_2z^5-96z^7-40x_1x_2z^4\\
&- 64x_2^2z^4-16x_2y_1z^4-16y_1^2z^4-28x_2y_2z^4+12y_1y_2z^4+28z^6-32 x_1x_2z^3\\
&-48x_2^2z^3-64x_2y_1z^3+16y_1^2z^3+4x_2y_2z^3+8y_1y_2z^3+48 z^5-12x_1x_2z^2\\
&+64x_2^2z^2+16y_1^2z^2+16x_2y_2z^2-2y_1y_2z^2+18z^4+16x_1x_2z+8y_1^2z\\
&-4y_1y_2z-24z^3-2y_1y_2-2z^2+1, \\
b' =& -144x_1x_2z^9+144x_2y_2z^9-72x_1^2z^8+384x_1x_2z^8+72x_1y_2z^8-240x_2 y_2z^8+192x_1^2z^7\\
&- 112x_1x_2z^7-120x_1y_2z^7-128x_2y_2z^7-92x_1^2z^6 -384x_1x_2z^6-28x_1y_2z^6\\
&+256x_2y_2z^6-96x_1^2z^5+208x_1x_2z^5+68x_1 y_2z^5+48x_2y_2z^5-24y_1y_2z^5\\
&-24z^7+40x_1^2z^4+64x_1x_2z^4+28x_1y_2 z^4-16x_2y_2z^4-16y_1y_2z^4\\
&+12y_2^2z^4+32z^6+32x_1^2z^3+48x_1x_2z^3- 4x_1y_2z^3-64x_2y_2z^3\\
&+16y_1y_2z^3+8y_2^2z^3+24z^5+12x_1^2z^2-64x_1 x_2z^2-16x_1y_2z^2\\
&+16y_1y_2z^2-2y_2^2z^2-32z^4-16x_1^2z+8y_1y_2z-4 y_2^2z-4z^3-2y_2^2+4z.
\end{align*}
complex realizes to the suspension of the Hopf map.
\end{thmintro}

At present we do not know whether the maps in Theorems 1 and 2 are homotopic over $\Z$. Both reduce to the identity over $\mathbb{F}_2$, since the generator of $\mathrm{GW}_1^0(k)$ in the first construction and $-1$ in the second become trivial modulo $2$. Finally, we note that Turiel constructs in \cite{Turiel07} a suspension procedure for certain maps between the algebraic spheres $\mathbb{S}_k^n=\Spec(k[x_1,\ldots,x_{n+1}]/(\Sigma x_i^2-1))$, including the Hopf map. Since $k[\mathbb{S}^{n}_\cplx]\cong k[Q_n]$, this method also gives a map $Q_4\to Q_3$, defined over $\Z[i]$, whose complex realization is the suspension of the Hopf map. The resulting unimodular row $(a^{\prime\prime},b^{\prime\prime})$ is

\begin{align*}
a^{\prime\prime} =& -2x_1^2z^2i-4x_1x_2z^2i-2x_2^2z^2i+2y_1^2z^2i-4y_1y_2z^2i+2y_2^2z^2i-4x_2y_1z^2+4x_1y_2z^2\\
&+8x_2y_2z^2+4z^4+2x_1^2zi+4x_1x_2zi+2x_2^2zi-2y_1^2zi+4y_1y_2zi-2y_2^2zi+4x_2y_1z\\
&-4x_1y_2z-8x_2y_2z-8z^3+x_1^2i+4x_1x_2i+x_2^2i-y_1^2i+4y_1y_2i-y_2^2i\\
&+2x_2y_1-2x_1y_2-8x_2y_2+4z,\\
b^{\prime\prime} =& 2x_1^2z^2i+4x_1x_2z^2i+2x_2^2z^2i+2y_1^2z^2i-4y_1y_2z^2i+2y_2^2z^2i-2x_1^2zi-4x_1x_2zi\\
&-2x_2^2zi-2y_1^2zi+4y_1y_2zi-2y_2^2zi+4z^3i-x_1^2i-4x_1x_2i-x_2^2i\\
&-y_1^2i+4y_1y_2i-y_2^2i-6z^2i+i.
\end{align*}
In Section \ref{sec:pone-suspension}, we develop a method for $\pone$-suspending maps between quadrics. Given a map $f \in \hom_{\mathrm{Sch}_k}(Q_i, Q_j)$, we produce a new map in $\hom_{\mathrm{Sch}_k}(Q_{i+2}, Q_{j+2})$ which has the homotopy class of $\Sigma_{\pone}f$. We use this to give a procedure for constructing a map representing any homotopy class in $[Q_i,Q_i]_{\A^1} \cong \mathrm{GW}(k)$ for $i \geq3$.

In the second part of the paper, we apply our construction of the exotic Hopf map to obtain an explicit rank $2$ vector bundle on a convenient replacement of $\mathbb{P}^3$. For context, recall that Atiyah and Rees observed in \cite{Atiyah76} that the bundle coming from the composite 
\[
\cplx\Proj^3 \to S^6 \xrightarrow{\Sigma^3 \eta }  S^5 \to BU(2)
\]
had trivial Chern classes, but was indecomposable. They then showed that all topological bundles on $\cplx \Proj^3$ are uniquely determined by their Chern classes and a $\Z_2$-invariant they called the $\alpha$-invariant. Horrocks \cite{Horrocks68} came up with a way of constructing rank $2$ bundles on $\cplx\Proj^3$ and Atiyah and Rees proved that all complex rank $2$ topological vector bundles on $\cplx\Proj^3$ can be obtained by starting with a sum of line bundles and then iteratively tensoring with a line bundle and applying Horrocks' construction \cite[Theorem 1.1]{Atiyah76}. In \cite{Asok20b}, it was shown that the same story holds in (unstable) motivic homotopy theory. Namely, there is a composite
\[
\Proj^3_k \to (\pone_k)^{\wedge 3} \xrightarrow{\Sigma_{\Proj^1}\eta' }  \A^3_k\smallsetminus\{0\} \to \mathrm{BSL}_2
\]
where the first map is the ``clutching'' map, the second one is the class of any of the exotic Hopf maps constructed above and the third one is the map classifying the universal rank $2$ bundle on $ \A^3_k\smallsetminus\{0\}$. However, by construction this composite yields a class in $[\Proj^3_k,\mathrm{BSL}_2]_{\A^1}$, i.e., a motivic vector bundle on $\mathbb{P}^3_k$. Such a bundle can be seen as an actual algebraic vector bundle on a convenient affine replacement of $\mathbb{P}^3_k$, the so-called \emph{Jouanolou device} of $\mathbb{P}^3_k$ denoted by $\mathbb{J}^3_k$ below. Up to homotopy, the above composite then reads as
\[
\mathbb{J}^3_k \to Q_6 \xrightarrow{\Sigma_{\Proj^1}\eta' }  Q_5 \to \mathrm{BSL}_2
\]
where the first map is an explicit model of the clutching map provided in Section \ref{sec:collapsemap} below, and the results in the first part of the paper allow us to obtain an explicit idempotent $3\times 3$-matrix of rank $2$ on $\mathbb{J}^3_k$ that is the algebraic construction of the topological composite above.  If $k=\cplx$, we prove that it has the expected algebraic invariants: trivial Chern classes and non-trivial Atiyah-Rees invariant. In view of Horrock's construction, the vector bundle on $\mathbb{J}^3_k$ is pulled-back from $\Proj^3_k$ in case $k=\cplx$, but we don't know if this is the case over an arbitrary field.

\subsection*{Conventions}
We work over a perfect field $k$. We use the conventions of \cite{Asok20b} for motivic homotopy theory, and the notation of \cite[\S 3]{Asok12c} for higher Grothendieck-Witt sheaves (a.k.a. Hermitian $K$-theory sheaves). In particular, we will use symplectic $K$-theory and denote the usual skew-symmetric matrix on $\Z^2$ by $H$, i.e.,
\[
H := \begin{pmatrix}
    0 & 1 \\ -1 & 0
\end{pmatrix}.
\]

\section{Maps between quadrics}

\subsection{Basics on quadrics}\label{sec:basicsquadrics}

We start by exploring a bit the geometry of $Q_{2n}$ for $n\geq 1$. Recall first from \cite[\S 2.1]{Asok16} that a morphism from an affine $k$-scheme $X=\mathrm{Spec}(R)$ to $Q_{2n}$ is determined up to homotopy by an ideal $I:=\langle a_1,\ldots,a_n,s\rangle\subset R$ such that $I/I^2$ is generated by the classes of $a_1,\ldots,a_n$. Using this, one sees that $s-s^2\in \langle a_1,\ldots,a_n\rangle$, i.e. there exists $b_1,\ldots,b_n\in R$ such that $s(1-s)=\sum_{i=1}^na_ib_i$ and two choices of $b_1,\ldots,b_n$ are homotopic. Granted this, we will mostly define morphisms to $Q_{2n}$ using explicit descriptions of ideals.

We consider the closed subschemes 
\[
Z_0:=\{x_1=\ldots=x_n=z=0\}\subset Q_{2n}, \phantom{i}Z_1:=\{x_1=\ldots=x_n=1-z=0\}\subset Q_{2n}
\]
and set $X_i:=Q_{2n}\smallsetminus Z_i$ for $i=0,1$. It is straightforward to check that $Z_i\simeq \A^n$ (with coordinates $y_1,\ldots,y_n$) and we know from \cite[Theorem 3.1.1]{Asok14} that the subschemes $X_i$ are $\A^1$-contractible. On the other hand
\[
X_0\cap X_1=Q_{2n}\smallsetminus (Z_0\sqcup Z_1)=Q_{2n}\smallsetminus \{x_1=\ldots=x_n=0\}
\]
and there is an obvious projection morphism 
\[
p_n\colon X_0\cap X_1\to \A^n\smallsetminus \{0\},\phantom{i}(x_1,\ldots,x_n,y_1,\ldots,y_n,z)\mapsto (x_1,\ldots,x_n).
\]
\begin{lem}
The map $p_n\colon  X_0\cap X_1\to \A^n\smallsetminus \{0\}$ is a weak equivalence. 
\end{lem}

\begin{proof}
We can cover $\A^n\smallsetminus\{0\}$ by the open subschemes $D(x_i)=\{x_i\neq 0\}\subset \A^n\smallsetminus\{0\}$. The preimage of $D(x_i)$ under $p_n$ is an affine space over $D(x_i)$ with coordinates $y_1,\ldots,y_{i-1},y_{i+1},\ldots,y_n,z$. 
\end{proof}

We can construct an inverse (up to homotopy) of that map, using the morphism
\[
\alpha_n\colon Q_{2n-1}\times \A^1\to Q_{2n}
\]
given on coordinates by $(x_1,\ldots,x_n,y_1,\ldots,y_n,t)\mapsto (x_1,\ldots,x_n,y_1t(1-t),\ldots,y_nt(1-t),t)$. A direct computation shows that it factors through $X_0\cap X_1$, and that the composite 
\[
Q_{2n-1}\times \A^1\xrightarrow{\alpha_n} X_0\cap X_1\xrightarrow{p_n} \A^n\smallsetminus \{0\}
\]
is just the projection map (which is a weak-equivalence).
For further use, we also note that $\alpha_n$ induces an isomorphism 
\[
\alpha_n\colon Q_{2n-1}\times (\A^1\smallsetminus \{0,1\})\simeq D(z(1-z))\subset Q_{2n}
\]
with inverse given $(x_1,\ldots,x_n,y_1,\ldots,y_n,z)\mapsto (x_1,\ldots,x_n,\frac {y_1}{z(1-z)},\ldots,\frac {y_n}{z(1-z)},z)$. 

\subsection{The universal rank $2$ bundle on $Q_4$}\label{sec:universal2bundle}

Recall from \cite{Panin21} that there is an isomorphism $Q_4\simeq \mathrm{H}\Proj^1$, where we see the right hand side as the open subscheme of the Grassmannian $\mathrm{Gr}(2,4)$ over which the symplectic form $H\perp H$ restricts to a nondegenerate symplectic form on the universal rank $2$ sub-bundle. We consider the Plücker embedding 
\[
\mathrm{Gr}(2,4)\to \mathbb{P}^5
\]
At a point $x\in  \mathrm{Gr}(2,4)$ given by a matrix
\[
0\to k(x)^2\xrightarrow{\begin{pmatrix} a_1 & b_1 \\ a_2 & b_2 \\ a_3 & b_3 \\ a_4 & b_4\end{pmatrix}} k(x)^4
\]
the condition to be non degenerate is given by $a_1b_2-a_2b_1+a_3b_4-a_4b_3\neq 0$. Using the Plücker coordinates 
\[
(u_1,u_2,u_3,u_4,u_5,u_6)=(a_1b_2-a_2b_1,a_1b_3-a_3b_1,a_1b_4-a_4b_1,a_2b_3-a_3b_2,a_2b_4-a_4b_2,a_3b_4-a_4b_3)
\]
this corresponds to $u_1+u_6\neq 0$. Setting $d:=u_1+u_6$, and using the coordinates 
\[
z:=\frac {a_1b_2-a_2b_1}d,x_1:=\frac {a_1b_3-a_3b_1}d, x_2:=\frac {a_1b_4-a_4b_1}d, y_1:=\frac {a_2b_4-a_4b_2}d, y_2:=\frac {a_3b_2-a_2b_3}d
\]
we obtain that the (unique) equation describing $\mathrm{Gr}(2,4)\cap D_+(d)$ is given by $x_1y_1+x_2y_2=z(1-z)$. 

Using this, we can describe the canonical rank $2$ bundle (and its canonical rank $2$ orthogonal) on $Q_4\simeq \mathrm{H}\Proj^1$. We will do it under the form of an idempotent matrix, that we now make explicit. We consider the commutative diagram
\[
\xymatrix{P\ar[r]^-{i}\ar[d]_-{\varphi} & R^4\ar[d]^-{H\perp H} \\
P^\vee & (R^4)^\vee\ar[l]^-{i^\vee}}
\]
where $R$ is the coordinate ring of $Q_4$, $P$ is the universal rank $2$ bundle and $\varphi\colon P\to P^\vee$ is the symplectic form induced by $H\perp H$. We then get an idempotent matrix using $i\varphi^{-1}i^\vee(H\perp H)$. 

At a point $x\in  \mathrm{Gr}(2,4)$ given by the above matrix $\begin{pmatrix} a_1 & b_1 \\ a_2 & b_2 \\ a_3 & b_3 \\ a_4 & b_4\end{pmatrix}$, a direct computation yields a matrix of the form
\[
\frac 1{d}\begin{pmatrix} a_1b_2-a_2b_1 & 0 & a_1b_4-a_4b_1 & a_3b_1-a_1b_3 \\ 0 & a_1b_2-a_2b_1 & a_2b_4-a_4b_2 & a_3b_2-a_2b_3 \\
a_3b_2-a_2b_3 & a_1b_3-a_3b_1 & a_3b_4-a_4b_3 & 0 \\ a_4b_2-a_2b_4 & a_1b_4-a_4b_1 & 0 & a_3b_4-a_4b_3\end{pmatrix}
\]
which corresponds to the matrix
\[
M:=\begin{pmatrix} z & 0 & x_2 & -x_1 \\ 0 & z & y_1 & y_2 \\
y_2 & x_1 & 1-z & 0 \\ -y_1 & x_2 & 0 & 1-z\end{pmatrix}.
\]
We let the reader check that this module coincides with the one in \cite[Example 4.3.8]{Asok14}.
The orthogonal of the rank $2$ universal subbundle is then given by the matrix
\[
N:=\mathrm{Id}-M=\begin{pmatrix} 1-z & 0 & -x_2 & x_1 \\ 0 & 1-z & -y_1 & -y_2 \\
-y_2 & -x_1 & z & 0 \\ y_1 & -x_2 & 0 & z\end{pmatrix}.
\]

\subsection{The Jouanolou device of $\mathbb{P}^n$}

Recall that the usual Jouanolou device $\mathbb{J}^n$ of $\Proj^n$ is given by the functor associating to a $k$-algebra $R$ the set of matrices $P\in M_{n+1}(R)$ of rank one satisfying $P^2=P$. We have an explicit presentation of the algebra representing this functor via
\[
k[\mathbb{J}^n]=k[x_{ij}\vert 1\leq i,j\leq n+1]/\left ( x_{ij}=\sum_{r=1}^{n+1} x_{ir}x_{rj}\forall 1\leq i,j\leq n+1, D\right)
\]
where $D$ is the set of all $2\times 2$ minors of $P$.
There is a universal matrix $P=(x_{ij})$ on $k[\mathbb{J}^n]$ which satisfies $P^2=P$ by the first set of equations. This yields a projector $P_R$ on $R^{n+1}$ for any homomorphism of $k$-algebras $k[\mathbb{J}^n]\to R$ and we set $L=\mathrm{Im}(P)\subset R^{n+1}$. This is a projective module and we can compute its rank by considering residue fields of $R$. If $F$ is a field we know that the matrix $P_F$ satisfies $P^2_F=P_F$ and it follows in particular that $P_F$ is diagonalizable with possible eigenvalues $0$ and $1$. The condition on the minors forces $P_F$ to be of rank one, and consequently $L$ is of rank one as required. 

\begin{rem}
If $k$ is of characteristic $0$, we can replace the conditions $D=0$ by the easier condition $\sum_{i=1}^{n+1}x_{ii}=1$, i.e. the matrix is of trace $1$.  
\end{rem}

We have an obvious morphism of schemes 
\[
p:\mathbb{J}^n\to \Proj^n
\]
corresponding to the line bundle $L$ and its generating sections given by $k[\mathbb{J}^n]^{n+1}\xrightarrow{P} L$. Concretely, let $c_1,\ldots,c_{n+1}$ (resp. $l_1,\ldots,l_{n+1}$) be the columns (resp. rows) of $P$. We set $U_i=\Jo^n\smallsetminus V(c_i)$, where $V(c_i):=V(x_{1i},\ldots,x_{(n+1)i})$ and we observe that $c_i$ is non-trivial on $U_i$, i.e. $c_i$ generates $L_{\vert U_i}$. On $U_i$, there exists then global sections $\lambda_{i,j}$ for $j=1,\ldots,n+1$ with $c_j=\lambda_{i,j}c_i$ and we obtain a map
\[
p_i:U_i\to \A^n
\]
given by the global sections $\lambda_{i,j}$, $i\neq j$. These maps glue together to yield the morphism $p:\mathbb{J}^n\to \Proj^n$. It turns out that $U_i$ is in fact isomorphic to $\A^{2n}$. To see this, we observe that we can write 
\[
P_{\vert U_i}=c_i\cdot \begin{pmatrix} \lambda_{i,1} & \ldots & \lambda_{i,i-1} & 1 & \lambda_{i,i+1} & \ldots & \lambda_{i,n+1}\end{pmatrix}.
\]
From $P^2=P$, we draw in particular that $P\cdot c_i=c_i$ and thus 
\[
\begin{pmatrix} \lambda_{i,1} & \ldots & \lambda_{i,i-1} & 1 & \lambda_{i,i+1} & \ldots & \lambda_{i,n+1}\end{pmatrix}\cdot c_i=1
\]
which reads as 
\[
\sum_{j=1, j\neq i}^{n+1}\lambda_{i,j}x_{ji}+x_{ii}=1.
\]
and $x_{ii}$ is determined by the other elements. This allows to define an isomorphism $\varphi_i:\A^{2n}\to U_i$ via
\[
(a_1,\ldots,a_n,b_1,\ldots,b_n)\mapsto \begin{pmatrix} a_1 \\ \vdots \\ a_{i-1} \\ 1-\sum_{j=1}^na_jb_j \\ a_i \\ \vdots \\ a_n\end{pmatrix}\cdot 
\begin{pmatrix} b_1 & \cdots & b_{i-1} & 1 & b_i & \cdots & b_n\end{pmatrix}.
\]
The point $0:=[0:\ldots:0:1]\in \mathbb{P}^n$ is contained in the image of the projection 
\[
p_{n+1}\colon U_{n+1}\to \A^n\subset \mathbb{P}^n
\]
and using $\varphi_{2n+1}$ we see that its preimage is of the form
\[
p_{n+1}^{-1}(\{0\})=\{ \begin{pmatrix} 0 & \hdots & 0 & a_1 \\ \vdots & \ddots & \vdots & \vdots \\ 0 & \hdots & 0 & a_n \\ 0 & \hdots & 0 & 1\end{pmatrix}\vert (a_1,\ldots,a_n)\in \A^n\}\simeq \A^n.
\] 
and we obtain a Cartesian square 
\begin{equation}\label{eqn:compatiblestructures}
\xymatrix{\A^n\ar[r]\ar[d] & \mathbb{J}^n\ar[d] \\
\{0\}\ar[r] & \mathbb{P}^n}
\end{equation}

\subsection{The map $\mathbb{J}^n\to Q_{2n}$}\label{sec:collapsemap}

Consider the ideal 
\[
I=\langle x_{ij}\vert 1\leq i\leq n+1, 1\leq j\leq n\rangle\subset k[\mathbb{J}^n].
\]
We claim that 
\[
I=\langle x_{(n+1)1},\ldots x_{(n+1),n},\sum_{i=1}^n x_{ii}\rangle=\langle x_{(n+1)1},\ldots x_{(n+1),n},1-x_{(n+1),(n+1)}\rangle=p_{n+1}^{-1}(\{0\}).
\]
 Indeed, for $1\leq i,j\leq n$ the $2\times 2$-minor
\[
\det \begin{pmatrix} x_{ij} & x_{i(n+1)} \\ x_{(n+1)j} & x_{(n+1)(n+1)}\end{pmatrix}
\]
is trivial, yielding
\[
x_{ij}(1-\sum_{i=1}^nx_{ii})=x_{ij}x_{(n+1)(n+1)}=x_{i(n+1)}x_{(n+1)j}.
\]
Thus $x_{ij}=x_{i(n+1)}x_{(n+1)j}+x_{ij}(\sum_{i=1}^nx_{ii})$ and the claim follows. Using again
\[
\det \begin{pmatrix} x_{ii} & x_{i(n+1)} \\ x_{(n+1)i} & x_{(n+1)(n+1)}\end{pmatrix}=0
\]
and summing up, we find
\[
\sum_{i=1}^nx_{i(n+1)}x_{(n+1)i}=(\sum_{i=1}^nx_{ii})x_{(n+1)(n+1)}=(\sum_{i=1}^nx_{ii})(1-\sum_{i=1}^nx_{ii})
\]
and it follows that $I/I^2$ is generated by the classes of $x_{(n+1)1},\ldots x_{(n+1),n}$. This yields a morphism 
\[
\pi_n\colon \mathbb{J}^n\to Q_{2n}
\] 
given by $(x_{ij})\mapsto (x_{(n+1)1},\ldots x_{(n+1),n},x_{1(n+1)},\ldots,x_{n(n+1)},\sum_{i=1}^nx_{ii})$. Moreover, the Cartesian square \eqref{eqn:compatiblestructures} yields a Cartesian square
\begin{equation}\label{eqn:commutativeinvariant}
\xymatrix{\mathbb{J}^n\smallsetminus \{\A^n\}\ar[r]\ar[d] & \mathbb{J}^n\ar[d] \\ 
\mathbb{P}^n\smallsetminus \{0\} \ar[r] & \mathbb{P}^n}
\end{equation}
the induced morphism on cofibers (using the purity isomorphism) reads as a morphism 
\[
\mathbb{J}^n\to (\A^n)_+\wedge (\pone)^{\wedge n}\to (\pone)^{\wedge n}
\]
for which $\pi_n$ is an explicit model.

\subsection{An explicit $\pone$-suspension map of maps between quadrics}
\label{sec:pone-suspension}
In this section, we provide an explicit way of $\pone$-suspending morphisms of quadrics, starting with the following observation: For $n\geq 1$, consider  $\A^{2n+2}$ with variables $x_1,\ldots,x_n,y_1,\ldots,y_n,t,u$. We have $f_n:=\sum_{i=1}^n x_iy_i-1$ and $g=tu-1$. These polynomials are irreducible, hence prime in $k[\A^{2n+2}]$, and
\[
V(f_n)=Q_{2n-1}\times \A^2, \phantom{i}V(g)=\A^{2n}\times \gm{}.
\]  
Since $f_n$ and $g$ are prime, we have $(f_n)\cap (g)=(f_n\cdot g)$ and obtain in view of \cite[Thm. 3.4, Thm. 3.5, Cor. 3.9]{Schwede05} a cocartesian square of schemes (whose arrows are the closed immersions given by the above polynomials)
\[
\xymatrix{Q_{2n-1}\times \gm{}\ar[r]\ar[d] & Q_{2n-1}\times \A^2\ar[d] \\
\A^{2n}\times \gm{}\ar[r] & X_{2n+1}}
\]
where $X_{2n+1}=V(f_n\cdot g)\subset \A^{2n+2}$. Arguing as in \cite[Section 3, Lemma 1.6]{Morel99}, it is not hard to see that $X_{2n+1}$ is, in fact, a push-out in the category of Nisnevich sheaves of sets on $\mathrm{Sm}_k$. In fact, we may use \cite[\S 7.3]{Morel08} to see that it is a model of the reduced join of the corners of the diagram, which is  $\Sigma Q_{2n-1}\wedge \gm{}\simeq Q_{2n+1}$. It is slightly inconvenient to work with singular schemes, and we now provide an explicit weak equivalence $X_{2n+1}\to  Q_{2n+1}$. We have morphisms 
\[
\varphi_1:Q_{2n-1}\times \A^2\to Q_{2n+1}
\]
given by $(x_1,\ldots,x_n,y_1,\ldots,y_n,t,u)\mapsto (x_1,\ldots,x_n,t,y_1,\ldots,y_n,0)$
and 
\[
\varphi_2:\A^{2n}\times \gm{}\to Q_{2n+1}
\]
defined by $(x_1,\ldots,x_n,y_1,\ldots,y_n,t,u)\mapsto (x_1,\ldots,x_n,t,y_1,\ldots,y_n,u(1-\sum_{i=1}^nx_iy_i))$. It is straightforward to check that these maps induce a map 
\[
\varphi:X_{2n+1}\to Q_{2n+1}.
\]
Concretely, the morphism $\varphi$ is given by
\[
k[x_1,\ldots,x_{n+1},y_1,\ldots,y_{n+1}]/f_{n+1} \to k[x_1,\ldots,x_n,y_1,\ldots,y_n,u,t]/(f_n\cdot g)
\]
with $(x_1,\ldots,x_n,y_1,\ldots,y_n)\mapsto (x_1,\ldots,x_n,y_1,\ldots,y_n)$, $x_{n+1}\mapsto t$ and $y_{n+1}\mapsto uf_n$.  To prove that it is a weak equivalence, we observe that the projection $\pi\colon Q_{2n+1}\to \A^{n+1}\smallsetminus\{0\}$ onto the first $n+1$-coordinates is a weak equivalence since its fibers are affine spaces of dimension $n$.  We can consider the preimages of the open subschemes $\A^n\times \gm{}$ and $(\A^n\smallsetminus\{0\})\times \A^1$ denoted respectively by $D(x_{n+1})$ and $U_n$. For the same reasons as above, the projections $\pi\colon D(x_{n+1})\to \A^n\times \gm{}$ and $\pi\colon U_n\to (\A^n\smallsetminus\{0\})\times \A^1$ are weak equivalences, and the same holds on the intersection $U_n\cap D(x_n)$ which becomes weakly equivalent to $(\A^n\smallsetminus\{0\})\times \gm{}$. Now, the morphism 
\[
\varphi_1\colon Q_{2n-1}\times \A^2\to Q_{2n+1}
\]
factorizes through $U_n$ by definition. This induces a weak equivalence $\varphi_1\colon  Q_{2n-1}\times \A^2\to U_n$ since its composite with the projection  $\pi\colon U_n\to (\A^n\smallsetminus\{0\})\times \A^1$ corresponds to $(x_1,\ldots,x_n,y_1,\ldots,y_n,t,u)\mapsto (x_1,\ldots,x_n,t)$. The same argument shows that the morphism
\[
\varphi_2:\A^{2n}\times \gm{}\to Q_{2n+1}
\]
also induces a weak equivalence $\A^{2n}\times \gm{}\to D(x_{n+1})$, while both morphisms induce weak equivalences $ Q_{2n-1}\times\gm{}\to  U_n\cap D(x_{n+1})$. Using  \cite[Section 3, Lemma 1.6]{Morel99}, we obtain a push-out square in $\mathcal{H}(k)$ of the form
\[
\xymatrix{U_n\cap D(x_{n+1})\ar[r]\ar[d] &  D(x_{n+1})\ar[d] \\
 U_n\ar[r] & Q_{2n+1}}
\] 
and the previous arguments show that $\varphi\colon X_{2n+1}\to Q_{2n+1}$ is a weak equivalence.

We proceed in an analogous manner for $Q_{2n}$. The details are as follows. For $n\geq 1$, consider $\A^{2n+3}$ with variables $x_1,\ldots,x_n,y_1,\ldots,y_n,z,t,u$. Set $h_n=\sum_{i=1}^nx_iy_i-z(1-z)$ and $g=tu-1$. The same argument as before yields a cocartesian square 
\[
\xymatrix{Q_{2n}\times \gm{}\ar[r]\ar[d] & Q_{2n}\times \A^2\ar[d] \\
\A^{2n+1}\times \gm{}\ar[r] & Y_{2n+2}}
\]
where $Y_{2n+2}=V(h_n\cdot g)\subset \A^{2n+3}$. We have morphisms
\[
\psi_1:Q_{2n}\times \A^2\to Q_{2n+2}
\]
given by $(x_1,\ldots,x_n,y_1,\ldots,y_n,z,t,u)\mapsto (x_1,\ldots,x_n,t,y_1,\ldots,y_n,0,z)$
and 
\[
\psi_2:\A^{2n+1}\times \gm{}\to Q_{2n+2}
\]
defined by $(x_1,\ldots,x_n,y_1,\ldots,y_n,z,t,u)\mapsto (x_1,\ldots,x_n,t,y_1,\ldots,y_n,u(z(1-z)-\sum_{i=1}^nx_iy_i),z)$. They induce a morphism 
\[
\psi:Y_{2n+2}\to Q_{2n+2}
\]
which is concretely given by 
\[
k[x_1,\ldots,x_{n+1},y_1,\ldots,y_{n+1},z]/h_{n+1}\to k[x_1,\ldots,x_{n},y_1,\ldots,y_{n},z,t,u]/(h_{n}\cdot g)
\]
with $(x_1,\ldots,x_{n},y_1,\ldots,y_{n},z)\mapsto (x_1,\ldots,x_{n},y_1,\ldots,y_{n},z)$, $x_{n+1}\mapsto t$ and $y_{n+1}\mapsto u(z(1-z)-\sum_{i=1}^nx_iy_i)$.

To see that $\psi$ is a weak equivalence, observe that we have open contractible subschemes $X_i\subset Q_{2n}$ for $i=0,1$. The morphism
\[
\psi_1:Q_{2n}\times \A^2\to Q_{2n+2}
\]
induces a morphism $X_0\times \A^2\to Q_{2n+2}$, which is easily seen to factor through the contractible subscheme $X_0':=D(x_1,\ldots,x_{n+1},z)\subset Q_{2n+2}$. Similarly, the morphism 
\[
\psi_2:\A^{2n+1}\times \gm{}\to Q_{2n+2}
\]
factors through $X_0'$ and we obtain a morphism from the push-out $P_0$ of the diagram
\[
\xymatrix{X_0\times \gm{}\ar[r]\ar[d] & X_0\times \A^2 \\
\A^{2n+1}\times \gm{} &  }
\]
to $X_0'$. Since $X_0$ and $\A^{2n+1}$ are both contractible, it is straightforward to check that the push-out is contractible and thus the induced map $Q_0\to X_0'$ is a weak equivalence. The same arguments apply replacing $X_0$ by $X_1$ obtaining a push-out $P_1$. Taking the intersection we thus obtain a morphism of the push-out $P_{01}$ of the diagram 
\[
\xymatrix{(X_0\cap X_1)\times \gm{}\ar[r]\ar[d] & (X_0\cap X_1)\times \A^2 \\
\A^{2n+1}\times \gm{} & }
\]
to $X'_{0}\cap X'_1$. The push-out of the diagram involving $P_{01},P_0,P_1$ is $Y_{2n+2}$, while the push-out of the diagram involving $X'_0\cap X'_1,X'_0,X'_1$ is $Q_{2n+2}$. This proves that the map from $Y_{2n+2}$ to $Q_{2n+2}$ is a weak equivalence. Once again, we think about $Y_{2n+2}$ as the singular $\pone$-suspension of $Q_{2n}$.

The construction described above allows to explicitly construct the $\pone$-suspension of a morphism $f\colon Q_{2n-1}\to Q_{2m}$ defined by mapping $(x_1,\ldots,x_n,y_1,\ldots,y_n)$ to $(f_1, \ldots, f_n, \tilde{f_1},\ldots, \tilde{f}_n, f_z)$. First note that an equivalent condition for $f$ being a morphism $Q_{2n-1}\to Q_{2m}$ is that the ideal 
\[
\langle f_1, \ldots, f_n, \sum_{i=1}^n x_iy_i -1 \rangle
\] 
in the ring $k[\A^{2n}]$ contains the element $f_z(1-f_z)$.

One observes that $f$ lifts to a map $F:\A^{2n}\to \A^{2m+1}$ given by the same formula, and to a map $(F)\times \mathrm{Id}_2:\A^{2n+2}\to \A^{2m+3}$ which descends to a map
\[
X_{2n+1}\to Y_{2m+2}
\]
and thus to a map $X_{2n+1}\to Q_{2m+2}$ which is given by mapping $(x_1,\ldots,x_n,y_1,\ldots,y_n,t,u)$ to 
\[
(f_1,\ldots f_n,t, \tilde f_1 ,\ldots\tilde f_n,\alpha,f_z)
\]
where $\alpha$ is an element in $k[X_{2n}]$ such that $\sum f_i\tilde f_i + t\alpha = f_z(1-f_z)$. Since $f\colon Q_{2n-1} \to Q_{2m}$ is a morphism, there exists an element $r$ such that $\sum f_i \tilde f_i + r(\sum x_i y_i -1) = f_z(1-f_z)$ in $k[\A^{2n}]$. By letting $\alpha = ur(\sum x_i y_i -1)$, we see that 
$\sum f_i\tilde f_i + t\alpha - r(tu-1)(\sum x_i y_i -1) = f_z(1-f_z)$ in $k[\A^{2n+2}].$

We then are left to find a factorization 
\[
X_{2n+1}\to Q_{2n+1}\to Q_{2m+2}
\]
of the above map. With a bit of thought, we realize that we may consider 
\[
Q_{2n+1}\to Q_{2m+2}
\]
given by the mapping $(x_1,\ldots, x_{n+1}, y_1, \ldots, y_{n+1})$ to 
\[
(f_1, \ldots, f_n, x_{n+1},\tilde{f}_1, \ldots,\tilde f_n, y_{n+1}r,  f_z)
\]
which does the job.

Doing the same arguments for maps $Q_{2n-1}\to Q_{2m-1}, Q_{2n}\to Q_{2m-1}$ and $Q_{2n}\to Q_{2m}$ give us the following results.

\begin{lem}
\label{lem:ponesuspoddeven}
The $\pone$-suspension of a map $f:Q_{2n-1}\to Q_{2m}$ (resp.$f:Q_{2n}\to Q_{2m})$  given by the ideal $\langle f_1, \ldots, f_m, f_z \rangle $ is the map $F:Q_{2n+1}\to Q_{2m+2}$ (resp. $F:Q_{2n+2}\to Q_{2m+2}$) given by the ideal $\langle f_1, \ldots, f_m, x_{n+1},f_z \rangle$.
\end{lem}

\begin{lem}
\label{lem:ponesuspevenodd}
The $\pone$-suspension of a map $f:Q_{2n}\to Q_{2m-1}$ (resp. $f:Q_{2n-1}\to Q_{2m-1}$) given by the unimodular row $(f_1, \ldots, f_m)$ is the map $F:Q_{2n+2}\to Q_{2m+1}$ (resp. $F:Q_{2n+1}\to Q_{2m+1}$) given by the unimodular row $(f_1, \ldots, f_n, x_{m+1})$.
\end{lem}

\begin{rem}
The maps of the two above lemmas are actually pointed, if we choose $(1,0\ldots,0,1,0,\ldots,0)$ as the base point of $Q_{2n+1}$ and $(0,\ldots,0)$ as the base point of $Q_{2m}$ in the first case, and $(1,0\ldots,0,1,0,\ldots,0)$ as the base point of $Q_{2n+1}$ and $(0,\ldots,0)$ as the base point of $Q_{2m}$ in the second case.
\end{rem}
\subsection{Endomorphisms of quadrics of arbitrary $\A^1$-Brouwer degree}
In this section we give a procedure for constructing morphisms $Q_n \to Q_n$ representing each element in $\mathrm{GW}(k)$, for $n\geq 2$  (stably if $n=2$). We begin by constructing the morphisms in the cases $[Q_2,Q_2]_{\A^1}$ and $[Q_3,Q_3]_{\A^1}$ and then suspend accordingly using Lemmas \ref{lem:ponesuspevenodd} and \ref{lem:ponesuspoddeven}. We start with considering endomorphisms of $Q_3$. 
\begin{lem}
    The map $Q_3 \to Q_3$ given by the unimodular row $(x_1, ux_2)$ represents the class $\langle u \rangle \in \mathrm{GW}(k)$ and the row $(y_1, -ux_2)$ represents the class $-\langle u \rangle \in \mathrm{GW}(k)$
\end{lem}

\begin{proof}
    For a unit $u\in k^\times$, the map $\A^2 \smallsetminus \{0\} \to \A^2 \smallsetminus \{0\}$ given on coordinates by $(x_1, x_2) \mapsto (x_1,ux_2)$ has $\A^1$-Brouwer degree $\langle u \rangle $ \cite[Lemma 9]{Kass19}. Since $Q_3 \simeq \A^2 \smallsetminus \{0\}$, this endomorphism lifts to the map $Q_3 \to Q_3$ given by the $\mathrm{SL}_2(k[Q_3])$-matrix. \[
    m_u = \begin{pmatrix}
        x_1 & ux_2\\ -u^{-1}y_2  & y_1
    \end{pmatrix}.
    \]
    The group structure on $[Q_3,Q_3]_{\A^1}$ is given by the group structure on $\mathrm{SL}_2$, hence the morphism representing $-\langle u\rangle $ is given by the inverse matrix 
    \[
   m_u^{-1} =  \begin{pmatrix}
        x_1 & ux_2\\ -u^{-1}y_2  & y_1
    \end{pmatrix}^{-1} =\begin{pmatrix}
        y_1 & -ux_2\\ u^{-1}y_2  & x_1
    \end{pmatrix} .
    \]
\end{proof}
Given an element in $q \in \mathrm{GW}(k)$, one does the following to find a morphism $Q_3 \to Q_3 $ with $\A^1$-Brouwer degree $q$:
\begin{enumerate}
    \item Write $q$ as a sum $q = \langle u_1, \ldots, u_r\rangle - \langle v_1, \ldots, v_s\rangle$. 
    \item Compute the matrix product $M_q = m_{u_1} \ldots m_{u_r} m_{v_1}^{-1} \ldots m_{v_s}^{-1}.$
    \item The first row of $M_q$ is a unimodular row with $\A^1$-Brouwer degree $q.$
\end{enumerate}
We will now consider $Q_2$-endomorphisms. Note that $Q_2 \simeq \mathbb{J}^1\simeq \pone$ and that in \cite{Barth25} the group structure on $[Q_2,\Proj^1]_{\A^1}$ is studied. The first complication is that $[Q_2,Q_2]_{\A^1}$ is not in the stable range, and we have $[Q_2,Q_2]_{\A^1} \cong \mathrm{GW}(k) \times _{k^\times / k^{\times 2}} k^\times$, but under stabilization, this map is just projection onto $\mathrm{GW}(k)$ (use for instance \cite[Theorem 6.37]{Morel08}. The group operation on $[Q_2,\Proj^1]_{\A^1}$ is in general not easy to describe, but the operation of adding a rank $0$ map to any other map can be described easily. We first begin by describing the generators of the rank $0$ maps. For $u\in k^\times$, we denote by $f_u\colon Q_2\to Q_2$ the morphism given by the ideal $\langle ux_1,z\rangle$.

\begin{lem} 
\label{lem:q2rank0maps}
    Let $u,v \in k^\times$, then the stable class of the morphism $f_{u,v}\colon Q_2 \to Q_2 $ given by the ideal $\langle \frac{(u-v)}{uv}x_1,  z + \frac{v}{u}(1-z)\rangle$ is $\langle u\rangle - \langle v\rangle$. 
\end{lem}
\begin{proof}
     It follows from \cite[Proposition 103]{Barth25} that the stable class of $f_u$ is $\langle u\rangle$, and from \cite[Lemma 91]{Barth25} that the class of $f_{u,v} \oplus f_v$ is $\langle u\rangle$. Hence the class of $f_{u,v}$ is $\langle u\rangle - \langle v\rangle $. 
\end{proof}
Since all $Q_2$-endomorphisms representing rank $0$ forms factor through $Q_3$ (\cite[Corollary 43]{Barth25}), the ideal described in Lemma \ref{lem:q2rank0maps} is in fact a unimodular row. Thus we define the $\mathrm{SL_2}(k[Q_2])$-matrix \[m_{u,v} := \begin{pmatrix}
    z + \frac{v}{u}(1-z) &  \frac{(u-v)}{uv}x_1 \\ (u-v)y_1  & z + \frac{u}{v}(1-z) 
\end{pmatrix}.\]
Consider polynomials $A_n(z), B_n(z)$ defined such that $(z + (1-z))^{2n} = z^n A_n(z) + (1-z)^n B_n(z)$. 
\begin{lem}
\label{lem:Q2hyperbolic}
    Let $n>0$ then: 
    \begin{enumerate}
        \item The map $Q_2 \to Q_2$ given by the ideal $\langle z^n A(n), x_1^nA_n(z)\rangle$ has degree 
    $\frac{n}{2}\langle 1,-1\rangle$ when $n$ is even and  $\frac{n-1}{2}\langle 1,-1\rangle + \langle1\rangle$ when $n$ is odd. 
    
    \item The map $Q_2 \to Q_2$ given by the ideal $\langle z^n A_n(z), y_1^nA_n(z)\rangle$ has degree 
    $-\frac{n}{2}\langle 1,-1\rangle$ when $n$ is even and  $-\frac{n-1}{2}\langle 1,-1\rangle - \langle -1\rangle$ when $n$ is odd.
    \end{enumerate}
\end{lem}
\begin{proof}
    Since these morphisms are defined over $\Z$, one can use a similar argument as the proof of Theorem 105 in \cite{Barth25} to show that the class the morphisms are determined by the Brouwer degree of their real and complex realizations. 
\end{proof}
 A $Q_2$-endomorphism can be described as a $(2 \times 2)$ idempotent matrix $P$ with trace $1$ and entries in $k[Q_2]$. For $n \geq 0$, we define the idempotent matrices
 \[
 P_n = \begin{pmatrix}
     z^n A_n(z) & x_1^nA_n(z) \\ y_1^n B_n(z) & (1-z)^nB_n(z)
 \end{pmatrix} \quad P_{-n}= \begin{pmatrix}
     z^n A_n(z) & y_1^nA_n(z) \\ x_1^n B_n(z) & (1-z)^nB_n(z)
 \end{pmatrix}.\]
In loc.cit. an action of rank $0$ maps on rank $n$ maps is described, and it is shown that this agrees with the group operation on $[Q_2, \pone]^{\A^1}$. One can check, following Remark 84 loc.cit, that the action of $m_{u,v}$ on $P$ is given by the matrix $m_{u,v}Pm_{u,v}^{-1}$.

Given an element in $q \in \mathrm{GW}(k)$, one does the following to find a morphism $Q_2 \to Q_2 $ that stably represents $q$. 
\begin{enumerate}
    \item If $q$ is of even rank $n$, write it as a sum $\frac{n}{2}\langle 1, -1\rangle + \sum_{i=1}^r \langle u_i\rangle - \langle v_i\rangle$ for some units $u_i, v_i \in k^\times$. If $q$ is of odd rank $n$, write it as a sum $\frac{n-1}{2}\langle 1, -1\rangle +\langle 1\rangle  + \sum_{i=1}^r \langle u_i\rangle - \langle v_i\rangle$ for some units $u_i, v_i \in k^\times$.
    \item One then computes matrix product $P_q= m_{u_1,v_1}\ldots m_{u_r,v_r}P_n m_{u_r,v_r}^{-1}\ldots m_{u_1,v_1}^{-1}$
    \item The first column of $P_q$ is the ideal defining the map $Q_2 \to Q_2$ with stable class $q$. 
\end{enumerate}
\begin{rem}
    The reason we do not decompose $q$ as a sum $n\langle 1\rangle + \sum_{i=1}^m \langle u_i\rangle - \langle v_i\rangle$ is that the $Q_2$-endomorphism representing $n\langle 1\rangle $ is a lot more complicated to write down compared to the maps given in Lemma \ref{lem:Q2hyperbolic}.
\end{rem}
\subsection{Motivic Hopf maps}

Given an H-space $X$, the Hopf construction detailed in \cite[\S 7.3]{Morel08} provides us with a fiber sequence 
\begin{equation*}
    X \to X\star X \to \Sigma X. 
\end{equation*}
Applying this to the algebraic group $\gm$ yields the fiber sequence
\begin{equation*}
    \gm{} \to Q_3 \to Q_2.
\end{equation*}
The map $\eta \colon Q_3 \to Q_2$ can be described as the rank $1$ projective $k[Q_3]$-module 
\begin{equation*}
    \begin{pmatrix}
        x_1y_1 & x_2y_1 \\ x_1y_2 & x_2y_2
    \end{pmatrix}.
\end{equation*}

\begin{prop}
\label{prop:etasusp}
The map $f\colon Q_3\to Q_2$ defined by $(x_1,x_2,y_1,y_2)\mapsto (x_2y_1,x_1y_2,x_1y_1)$ represents the Hopf map. For $n\geq 2$, the map $f\colon Q_{2n+1} \to Q_{2n}$ defined by mapping $(x_1,x_2,\ldots,x_{n+1},y_1,y_2,\ldots,y_{n+1})$ to $(x_2y_1,x_3, \ldots, x_{n+1},x_1y_2,y_3x_1y_1,\ldots,y_{n+1}x_1y_1,x_1y_1)$ is the $(n-1)$-fold $\mathbb{P}^1$-suspension of the Hopf map $\eta \colon Q_3 \to Q_2$.
\end{prop}

\begin{proof}
For the first claim, recall that we have a homotopy equivalence $\pi \colon Q_3 \to \A^2 \smallsetminus\{0\}$ sending $(x_1,x_2,y_1,y_2)$ to the point $(x_1,x_2)$ in $\A^2 \smallsetminus\{0\}$. There is also a homotopy equivalence $Q_2 \to \pone$ which maps a point $(x_1,y_1,z) \in Q_2$ to the point $[z:x_1]$ or $[y_1 : 1-z]$ (whichever is defined). The Hopf map $\eta\colon \A^2 \smallsetminus\{0\} \to \pone$ is defined by $(x_1,x_2) \to [x_1:x_2]$. It is straightforward to check that the square 
\begin{equation*}
    \xymatrix{Q_3\ar[r]^\eta\ar[d]_\pi & Q_2\ar[d]^\pi \\
\A^2 \smallsetminus\{0\} \ar[r]_{\eta} & \mathbb{P}^1.}
\end{equation*}
commutes. The second claim then follows from Lemma \ref{lem:ponesuspoddeven}.
\end{proof}

\begin{rem}
We could instead consider the map $g\colon Q_3\to Q_2$ defined by 
\[
(x_1,x_2,y_1,y_2)\mapsto (x_1x_2,y_1y_2,x_2y_2),
\]
which looks more symmetric (and could actually be used to obtain the so-called strong Euler class defined in \cite{Das15} and \cite{vdKallen15}). It differs from the previous one by the automorphism $a\colon Q_3 \to Q_3$ defined by $x_1\mapsto y_2,x_2\mapsto x_1,y_1\mapsto x_2,y_2\mapsto y_1$. Since 
\[
\begin{pmatrix} y_2 & x_1 \\ -y_1 & x_2\end{pmatrix}=\begin{pmatrix} x_2 & -x_1 \\ y_1 & y_2\end{pmatrix}^{-1} \text{ and }\begin{pmatrix} x_2 & -x_1 \\ y_1 & y_2\end{pmatrix}=\begin{pmatrix} x_1 & x_2 \\ -y_2 & y_1\end{pmatrix}\begin{pmatrix} 0 & -1 \\ 1 & 0\end{pmatrix}
\]
we see that the map $g\colon Q_3\to Q_2$ is homotopic to $-\eta$. 
\end{rem}

Using the Hopf construction for the algebraic group $\mathrm{SL}_2$ yields the fiber sequence
\begin{equation*}
    \mathrm{SL_2} \to Q_7 \xrightarrow{\nu} Q_4. 
\end{equation*}
As explained in \cite[\S3]{Fasel16a}, the scheme $Q_7$ can be seen of the scheme of pairs of $(2\times 2)$-matrices $(M_1,M_2)$ with $\det M_1 + \det M_2 = 1$. The map $\nu\colon Q_7 \to Q_4$ is given by $(M_1 M_2, \det(M_1))$, and its real realization is actually the Hopf map by \cite[Proposition 3.0.1]{Fasel16a}. Note that the complex realization of $\nu$ is the second Hopf map by \cite{Gluck86}. Using once again Lemma \ref{lem:ponesuspoddeven}, we obtain the following proposition.

\begin{prop}\label{prop:nususp}
Let $n\geq 2$. The map $f\colon Q_{2n+3} \to Q_{2n}$ defined by the ideal 
\begin{equation*}
    (-x_3y_1-x_2y_4,-x_4y_1+x_2y_3, x_5, \ldots, x_{n+2}, x_3y_3 + x_4y_4)
\end{equation*} 
is the $(n-2)$-fold $\mathbb{P}^1$-suspension of the Hopf map $\nu \colon Q_7 \to Q_4$.
\end{prop}

Complex realization of the maps from Proposition \ref{prop:etasusp} gives us complex algebraic representatives of the maps $\Sigma^{2n-2}\eta$ in the category of topological spaces. Similarly, the real realization of the maps from Proposition \ref{prop:nususp} gives us real algebraic representatives of the maps $\Sigma^{n-2}\eta$. Our two families of maps do not allow us to create a map whose complex realization is $\Sigma \eta$. Such a map would need to be a map $Q_4 \to Q_3$, and we construct such a map in the next section. 

\section{Exotic Hopf maps}

\subsection{First construction: Symplectic $K$-theory}\label{sec:symplecticHopfmap}
In this section, we construct a map $Q_4 \to Q_3$ that (complex) realizes to the suspension of the Hopf map.  We first note that we have an exact sequence of sheaves \cite[Theorem 3.3]{Asok12a} 
\begin{equation}\label{eqn:pi2a2minus0}
\KMW4/12_\epsilon\to \piaone_{2}(Q_3)\to \mathbf{GW}_3^2\to 0,
\end{equation}
where $\mathbf{GW}_3^2$ is the Nisnevich sheafification of the presheaf $X\mapsto \mathrm{GW}_3^2(X)=\mathrm{KSp}_3(X)$. Contracting twice and using \cite[Proposition 2.9]{Asok12b}, \cite[Proposition 4.4]{Asok12a} we obtain an exact sequence
\[
\KMW2/12_\epsilon\to \piaone_{2,2}(Q_3)\to \mathbf{GW}_1^0\to 0.
\]
For any field $F/k$
\[
\mathbf{GW}_1^0(F)=\mathrm{KO}_1(F)=F^\times/(F^{\times})^2\times \Z/2\Z
\]
with generator for the $\Z/2\Z$ factor given by the permutation matrix $\begin{pmatrix}0 & 1 \\ 1 & 0\end{pmatrix}$.

On the other hand, $\mathbf{GW}_3^2=\piaone_2(\mathrm{Sp})$ and the right-hand map in \eqref{eqn:pi2a2minus0} is induced by the inclusion $Q_3\to \mathrm{Sp}$ (recall that $Q_3 = \mathrm{SL}_2 = \mathrm{Sp}_2$). Any map $Q_4\to Q_3$ can be composed with this inclusion, and conversely using the usual fiber sequences and connectivity arguments (e.g. \cite[Corollary 2.4]{Asok12a}), we see that any map $Q_4\to \mathrm{Sp}$ comes from a map $Q_4\to Q_3$ (non necessarily unique if $k$ is not algebraically closed). This leads us to consider 
\[
[Q_4,\mathrm{Sp}]_{\aone}=\mathrm{KSp}_1(Q_4)=\mathrm{GW}^2_1(Q_4).
\]
We can use the coniveau spectral sequence \cite[Definition 26]{Fasel09c} to compute this group, and we see that (using the fact that the cohomology groups of $Q_4$ with coefficients in strictly $\A^1$-invariant sheaves is easy to compute, see e.g. \cite[Proposition 1.1.5]{Asok16}) this spectral sequence yields a presentation of the form
\[
\mathrm{H}^0(Q_4,\mathbf{GW}_2^2)\to \mathrm{H}^2(Q_4,\mathbf{GW}_3^2)\to \mathrm{GW}_1^2(Q_4)\to 0
\] 
which reduces to 
\[
\KMW2(k)=\mathbf{GW}_2^2(k)\to \mathbf{GW}_1^0(k)\to \mathrm{GW}_1^2(Q_4)\to 0
\] 
If $k$ is algebraically closed, the left-hand term is isomorphic to the uniquely divisible group $\mathbf{K}^{\mathrm{M}}_2(k)$, while the group $ \mathbf{GW}_1^0(k)=\Z/2\Z$ showing that $\mathrm{GW}_1^2(Q_4)=\Z/2\Z$. Since $\Z/2\Z$ is invariant under field extensions, this shows that over any field the direct factor $\Z/2\Z$ of $\mathbf{GW}_1^0(k)$ survives in $\mathrm{GW}_1^2(Q_4)$.

To describe the image of $\Z/2\Z$ under the map $\mathbf{GW}_1^0(k)\to \mathrm{GW}_1^2(Q_4)$, we draw inspiration from \cite[Section 15.2]{Fasel08a}. The idempotent matrix
\[
M:=\begin{pmatrix} z & 0 & x_2 & -x_1 \\ 0 & z & y_1 & y_2 \\
y_2 & x_1 & 1-z & 0 \\ -y_1 & x_2 & 0 & 1-z\end{pmatrix}
\]
defined in Section \ref{sec:universal2bundle} above yields a rank $2$ projective module $P:=\mathrm{Im}(M)\subset k[Q_4]^4$, and we define a section $s\colon P\to k[Q_4]$ by considering the restriction to $P$ of the projection on the first factor $k[Q_4]^4\to k[Q_4]$. The image of $s$ is generated by the first row of $M$, and is thus the ideal $\mathfrak p:=\langle x_1,x_2,z\rangle$, which is prime of height $2$ with residue field $k(\mathfrak p)=k(y_1,y_2)$.  We have a commutative diagram
\[
\xymatrix{0\ar[r] & k[Q_4]\ar[r]\ar[d]_{-1} & P\ar[r]^-s\ar[d]^-\psi & k[Q_4]\ar[r]\ar@{=}[d] & k[Q_4]/\mathfrak p\ar[r]\ar[d]^-\varphi & 0 \\
0\ar[r] & k[Q_4]\ar[r] & P^\vee\ar[r]^-s & k[Q_4]\ar[r] & \mathrm{Ext}^2_{k[Q_4]}(k[Q_4]/\mathfrak p,k[Q_4])\ar[r] & 0}
\]
in which $\psi$ is the symplectic form induced by $(k[Q_4]^4,H^2)$ on $P$ and $\varphi$ is the symmetric form defined by mapping $1\in k[Q_4]/\mathfrak p$ to the top exact sequence (seen as an element in $ \mathrm{Ext}^2_{k[Q_4]}(k[Q_4]/\mathfrak p,k[Q_4])$). This symmetric form is actually a generator of $\widetilde{\mathrm{CH}}^2(Q_4)$ as shown in \cite[Lemma 4.2.6]{Asok14}.  The permutation matrix $\begin{pmatrix} 0 & 1 \\ 1 & 0\end{pmatrix}$ is orthogonal on $((k[Q_4]/\mathfrak p)^2,\varphi^{\oplus 2})$ and it is straightforward to check that it induces a commutative diagram on the direct sum of the above sequence with itself, corresponding to the permutation of the two factor of $P\oplus P$. This is obviously a symplectic transformation of $(P\oplus P,\psi\oplus\psi)$ and we can use it to define a symplectic matrix $U$ on $k[Q_4]^8$ by the commutative diagram
\[
\xymatrix{k[Q_4]^4\oplus k[Q_4]^4\ar[r]\ar[d]^-U & (P\oplus Q)\oplus (P\oplus Q)\ar[r] & (P\oplus P)\oplus (Q\oplus Q)\ar[d]^-{\tiny{\begin{pmatrix} 0 & 1 & 0 & 0 \\ 1 & 0 & 0 & 0 \\ 0 & 0 & 1 & 0\\ 0 & 0 & 0 & 1\end{pmatrix}}} \\
k[Q_4]^4\oplus k[Q_4]^4\ar[r] & (P\oplus Q)\oplus (P\oplus Q)\ar[r] & (P\oplus P)\oplus (Q\oplus Q)} 
\]
in which $Q$ is the orthogonal of $P$, i.e. the image of the idempotent matrix $N$ above.
To write $U$ explicitly, we observe that the isomorphism $k[Q_4]^4\to P\oplus Q$ is given by $x\mapsto (M(x),N(x))$, with inverse $P\oplus Q\to k[Q_4]^4$ induced by the respective inclusions. Thus
\[
U=\begin{pmatrix} 1-z & 0 & -x_2 & x_1 & z & 0 & x_2 & -x_1 \\ 0 & 1-z & -y_1 & -y_2 & 0 & z & y_1 & y_2 \\
-y_2 & -x_1 & z & 0 & y_2 & x_1 & 1-z & 0 \\ y_1 & -x_2 & 0 & z & -y_1 & x_2 & 0 & 1-z \\
z & 0 & x_2 & -x_1 & 1-z & 0 & -x_2 & x_1 \\ 0 & z & y_1 & y_2 & 0 & 1-z & -y_1 & -y_2 \\
y_2 & x_1 & 1-z & 0 & -y_2 & -x_1 & z & 0 \\ -y_1 & x_2 & 0 & 1-z & y_1 & -x_2 & 0 & z\end{pmatrix}.
\]

We will now reduce the matrix $U\in\mathrm{Sp}_8$ from the previous section into an element in $\mathrm{Sp}_2$ by using elementary symplectic operations. To make the following computations easier, we will need to switch to a different symplectic basis. Let $I_n$ be the $(n \times n)$-identity matrix, we define $J_n := \begin{pmatrix}
    0 & I_n \\ -I_n & 0
\end{pmatrix}$. Setting 
 \begin{equation*}
     S := \begin{pmatrix}
    1& 0& 0& 0& 0& 0& 0& 0\\ 0& 0& 1& 0& 0& 0& 0& 0\\ 0& 0& 0& 0& 1& 0& 0& 0\\ 0& 0& 0& 0& 0& 0& 1& 0\\ 0&
      1& 0& 0& 0& 0& 0& 0\\ 0& 0& 0& 1& 0& 0& 0& 0\\ 0& 0& 0& 0& 0& 1& 0& 0\\ 0& 0& 0& 0& 0& 0& 0& 1
\end{pmatrix},
\end{equation*}
we obtain 
$S \cdot \mathrm{diag(H,H,H,H)} \cdot S^T = J_4$, allowing us to change between the two symplectic bases. 

The elementary symplectic matrices with respect to $J_n$ come in two types. If $E$ is an $(n\times n)$-elementary matrix and $C$ is a \emph{symmetric} $(n\times n)$-matrix. The block matrices 
\begin{equation*}
    \begin{pmatrix}
        E & 0 \\ 0 & E^{-T}
    \end{pmatrix} \quad \text{and} \quad \begin{pmatrix}
        I_n & C \\ 0  & I_n
    \end{pmatrix}
\end{equation*} 
and their transposes are elementary symplectic matrices. On the other hand, the embedding $\mathrm{Sp}_{2n}\to \mathrm{Sp}_{2n+2}$ with respect to the inclusion $J_n \to J_{n+1}$ can be described as follows. The inclusion of the matrix $\begin{pmatrix}
    A & B \\ C  & D
\end{pmatrix} \in Sp_{2n}$ into $Sp_{2n+2}$ is the matrix 
\begin{equation*}
    \begin{pmatrix}
    A & 0_{n \times 1} & B  & 0_{n \times 1} \\ 0_{1 \times n} & 1 & 0_{1 \times n} & 0 \\ C & 0_{n \times 1}  & D & 0_{n \times 1} \\ 0_{1 \times n} & 0 & 0_{1 \times n} & 1
\end{pmatrix}.
\end{equation*}

We will now start with the symplectic reduction of the matrix
\begin{equation*}
    U=\begin{pmatrix} 1-z & 0 & -x_2 & x_1 & z & 0 & x_2 & -x_1 \\0 & 1-z & -y_1 & -y_2 & 0 & z & y_1 & y_2 \\
-y_2 & -x_1 & z & 0 & y_2 & x_1 & 1-z & 0 \\ y_1 & -x_2 & 0 & z & -y_1 & x_2 & 0 & 1-z \\
z & 0 & x_2 & -x_1 & 1-z & 0 & -x_2 & x_1 \\0 & z & y_1 & y_2 & 0 & 1-z & -y_1 & -y_2 \\
y_2 & x_1 & 1-z & 0 & -y_2 & -x_1 & z & 0 \\ -y_1 & x_2 & 0 & 1-z & y_1 & -x_2 & 0 & z\end{pmatrix}.
\end{equation*}
Performing the above change of basis, we see that $U$ becomes symplectic with respect to $J_4$ by conjugating with $S$:

\begin{equation*}
    M_1 := SUS^T = \begin{pmatrix}
        
    -z+1& -x_2& z& x_2& 0& x_1& 0& -x_1\\ -y_2& z& y_2& -z+1& -x_1& 0& x_1& 0\\ z& x_2& -z+1& -x_2& 0& -x_1& 0& x_1\\ y_2& -z+1& -y_2& z& x_1& 0& -x_1& 0\\ 0& -y_1& 0& y_1& -z+1& -y_2& z& y_2\\ y_1& 0& -y_1& 0& -x_2& z& x_2& -z+1\\ 0& y_1& 0& -y_1& z& y_2& -z+1& -y_2\\ -y_1& 0& y_1& 0& x_2& -z+1& -x_2& z \end{pmatrix}.
\end{equation*}
Setting
\begin{equation*}
    E_1 := \begin{pmatrix}
        1&0&0&0\\0&1&0&1\\0&0&1&0\\0&0&0&1
    \end{pmatrix}, \quad E_2 := \begin{pmatrix}
        1& 0& 0& 0\\ 0& 1& 0& 0\\ 0& 0& 1& 0\\ -y_2& z-1& y_2& 1
    \end{pmatrix},
\end{equation*}
the product 
\begin{equation*}
    \begin{pmatrix}
        E_1 & 0 \\ 0 & E_1^{-T}
    \end{pmatrix}^{-1}M\begin{pmatrix}
        E_1 & 0 \\ 0 & E_1^{-T}
    \end{pmatrix} \begin{pmatrix}
        E_2 & 0 \\ 0 & E_2^{-T} 
    \end{pmatrix}
\end{equation*}
yields the matrix 
\begin{equation*}
    \begin{pmatrix}
        -z+1& -x_2& z& 0& 0& 2x_1& 0& -2x_1z+x_1\\ -2y_2& 2z-1& 2y_2& 0& -2x_1& 0& 2x_1& -4x_1y_2\\ z& x_2& -z+1& 0& 0& -2x_1& 0& 2x_1z-x_1\\ 0& 0& 0& 1& x_1& 0& -x_1&2x_1y_2\\ 0& -y_1& 0& 0& -z+1& -2y_2& z& 0\\ y_1& 0& -y_1& 0& -x_2& 2z-1& x_2& -2x_2y_2-2z^2+2z\\ 0& y_1& 0& 0& z& 2y_2& -z+1& 0\\ 0& 0& 0& 0& 0& 0& 0& 1

.
    \end{pmatrix}.
\end{equation*}
multiplying from the right by $\begin{pmatrix}
    I_4 & S_1 \\ 0 & I_4
\end{pmatrix}$, where $S_4$ is the symmetric matrix 
\begin{equation*}
    S_4 = \begin{pmatrix}
        0&0&0&-x_1\\0&0&0&0\\0&0&0&x_1\\-x_1&0&x_1&-2x_1y_2
    \end{pmatrix}
\end{equation*}
we obtain 
\begin{equation*}
    \begin{pmatrix}
        -z+1& -x_2& z& 0& 0& 2x_1& 0& 0\\ -2y_2& 2z-1& 2y_2& 0& -2x_1& 0& 2x_1& 0\\ z& x_2& -z+1& 0& 0& -2x_1& 0& 0\\ 0& 0& 0& 1& 0& 0& 0& 0\\ 0& -y_1& 0& 0& -z+1& -2y_2&z& 0\\ y_1& 0& -y_1& 0& -x_2& 2z-1& x_2& 0\\ 0& y_1& 0& 0& z& 2y_2& -z+1& 0\\ 0& 0& 0& 0& 0& 0& 0& 1        
    \end{pmatrix},
\end{equation*}
from which we extract the $\mathrm{Sp}_6$ matrix 
\begin{equation*}
    M_2 := \begin{pmatrix}
        -z+1& -x_2& z& 0& 2x_1& 0\\ -2y_2& 2z-1& 2y_2& -2x_1& 0& 2x_1\\ z& x_2& -z+1& 0& -2x_1& 0\\ 0& -y_1& 0& -z+1& -2y_2& z\\ y_1& 0& -y_1& -x_2& 2z-1& x_2\\ 0& y_1& 0& z&2y_2& -z+1
    \end{pmatrix}.
\end{equation*}
Repeating the same procedure on $M_2$ with the elementary matrices 
\begin{equation*}
    E_3 := \begin{pmatrix}
        1 & 0 & 1 \\ 0 & 1 &0 \\ 0 & 0 & 1
    \end{pmatrix}, \quad E_4 := \begin{pmatrix}
        1 & 0 & 0 \\ 0 & 1 & 0 \\ -z & -x_2 & 1
    \end{pmatrix}
\end{equation*}
and the symmetric matrix
\begin{equation*}
   S_2 :=  \begin{pmatrix}
        0& 0& 0\\ 0& 0& 2x_1\\ 0& 2x_1& 2x_1x_2
    \end{pmatrix}
\end{equation*}
yields the $\mathrm{Sp}_4$ matrix
\begin{equation*}
    M_3 := \begin{pmatrix}
    -2z+1& -2x_2& 0& 4x_1\\ -2y_2& 2z-1& -4x_1& 0\\ 0& -y_1& -2z+1& -2y_2\\ y_1& 0& -2x_2& 2z-1
\end{pmatrix}.
\end{equation*}

Consider next the elementary matrices 
\begin{equation*}
    E_5 := \begin{pmatrix}
        1 & 2x_2 \\0 & 1 
    \end{pmatrix}, \quad E_6 := \begin{pmatrix} 1 & 0 \\ -y_1 & 1
    \end{pmatrix} 
\end{equation*}
and the symmetric matrices 
\begin{equation*}
    S_3 := \begin{pmatrix}
        0 & 0 \\ 0 & -1 
    \end{pmatrix}, \quad S_4 := \begin{pmatrix}
        0 & 0 \\ 0 & 1 - (2z-1).
    \end{pmatrix}
\end{equation*}
The product
\begin{equation*} M_4 := 
\begin{pmatrix} 
        E_5 & 0 \\ 0 & E_5^{-T}
    \end{pmatrix}
    \begin{pmatrix}
        I_2 & 0 \\ S_3 & I_2 
    \end{pmatrix} M_3 \begin{pmatrix}
        I_2 & S_4 \\ 0 & I_2 
    \end{pmatrix} \begin{pmatrix}
        E_6 & 0 \\ 0 & E_6^{-T}
    \end{pmatrix} 
\end{equation*}
yields the matrix 
\begin{equation*}
    M_4 = \begin{pmatrix}
       -4x_2y_1z+4x_2y_1-4x_2y_2-2z+1& 4x_2z-4x_2& -8x_1x_2& 8x_2^2y_2+8x_2z+4x_1-8x_2\\ -2y_1z+y_1-2y_2& 2z-1& -4x_1& 4x_2y_2+2z-2\\ y_1^2& -y_1& -2z+1& -y_1-2y_2\\-2x_2y_1^2+2y_1z+2y_2& 2x_2y_1-2z+1& 4x_2z+4x_1-4x_2& 1
    \end{pmatrix}.
\end{equation*}
Using the $1$ in the bottom right corner, we aim to use elementary transformation to remove the entries $-y_1 - 2y_2$ and $4x_2z + 4x_1 -4x_2$. Consider next the elementary matrices 
\begin{equation*}
    E_7 := \begin{pmatrix}
        1  & 0 \\ -y_1-2y_2 & 1
    \end{pmatrix}, \quad E_8 := \begin{pmatrix}
        1 & 4x_2z+4x_1-4x_2 \\ 0 & 1
    \end{pmatrix},
\end{equation*}
we set 
\begin{equation*}
    N := \begin{pmatrix}
        E_7 & 0 \\ 0 & E_7^{-T}
    \end{pmatrix}
    M_4\begin{pmatrix}
        E_8& 0 \\ 0 & E_8^{-T}
    \end{pmatrix} .
\end{equation*}
The matrix $N$ is too big to write down explicitly here, but it can be written as the matrix
\begin{equation*}
    \begin{pmatrix}
        & & & N_{14} \\
        & N'& & N_{24} \\
        & & & 0 \\
        N_{41} & N_{42} & 0 & 1
    \end{pmatrix}
\end{equation*}
where $N'$ is a $(3 \times 3)$-matrix. 
Using finally the symmetric matrices 
\begin{equation*}
    S_5 = \begin{pmatrix}
        0 & -N_{14} \\ -  N_{14} & -N_{24}
    \end{pmatrix},\quad S_6 = \begin{pmatrix}
        0 & -N_{41} \\ -  N_{41} & -N_{42}
    \end{pmatrix} .
\end{equation*} 
The product 
\begin{equation*}
    \begin{pmatrix}
        I_2 & S_5 \\ 0 & I_2
    \end{pmatrix}N \begin{pmatrix}
        I_2 & 0 \\ S_6 & I_2
    \end{pmatrix}
\end{equation*} 
is then the image of the $Sp_2 $ matrix $\begin{pmatrix}
    a & b \\ c & d
\end{pmatrix}$ into $Sp_4$ where 

\begin{align*}
    a =& 16x_2^3y_1^2y_2+16x_2^2y_1^2z-16x_2^2y_1y_2z-16x_2^2y_1^2-8x_2^2y_1y_2-16x_2^2y_2^2-24x_2y_1z^2 \\&+20x_2y_1z-8x_2y_2z+8z^3+4x_2y_1-8x_1y_2+12x_2y_2-8z^2-2z+1, \\
    b =& -32x_2^3y_2z-32x_1x_2^2y_2+32x_2^3y_2-32x_2^2z^2-48x_1x_2z+64x_2^2z-16x_1^2+40x_1x_2-32x_2^2, \\
    c =& -2x_2y_1^3-4x_2y_1^2y_2+2y_1^2z+4y_1y_2z+y_1^2+2y_1y_2+4y_2^2, \\
    d =& 4x_2y_1z+8x_2y_2z-4x_2y_1+8x_1y_2-12x_2y_2-4z^2+2z+1.
\end{align*}

\subsection{Second construction: weight shifting}

\subsubsection{The universal bundle on $Q_5$}
We consider the fiber sequence
\[
Q_5\to \mathrm{BSL}_2\to \mathrm{BSL}_3 
\]
in which the left-hand map classifies the universal rank $2$ bundle $P$ on $Q_5$. Writing 
\[
R=k[x_1,x_2,x_3,y_1,y_2,y_3]/(\sum_{i=1}^3 x_iy_i-1)
\]
for the coordinate ring of $Q_5$, it is given by the exact sequence
\[
0\to P\to R^3\xrightarrow{(x_1\phantom{i}x_2\phantom{i}x_3)} R\to 0,
\]
which is split by the homomorphism $R\xrightarrow{\begin{pmatrix} y_1 \\ y_2 \\y_3\end{pmatrix}}R^3$. We can then see $P$ as the image of the idempotent matrix
\begin{equation}\label{eqn:universalidempotent}
M(P):=\begin{pmatrix} 1-x_1y_1 & -x_2y_1 & -x_3y_1 \\ -x_1y_2 & 1-x_2y_2 & -x_3y_2 \\ -x_1y_3 & -x_2y_3 & 1-x_3y_3\end{pmatrix}.
\end{equation}
\begin{rem}
It is possible to show that the map $Q_5\to \mathrm{BSL}_2$ factorizes as $Q_5\to Q_4\to \mathrm{BSL}_2$, in which $Q_5\to Q_4$ is the $\pone$-suspension of the Hopf map $\eta\colon Q_3\to Q_2$ (Proposition \ref{prop:etasusp}) and $Q_4\to \mathrm{BSL}_2$ is the map classifying the universal rank $2$ symplectic bundle on $Q_4$. Since we don't use this fact, we leave the details to the reader.
\end{rem}

The scheme $\A^3\smallsetminus\{0\}$ can be covered by the two open subschemes $\A^2\smallsetminus\{0\}\times \A^1$ and $\A^2\times \gm{}$, and we transform this covering into a covering of $Q_5$ using the affine open subscheme $U:=\{x_3\neq 0\}\subset Q_5$ and the affine closed subscheme $V:=\{y_3=0\}\subset Q_5$. We obtain a Cartesian square 
\begin{equation}\label{eqn:homotopycartesianQ5}
\xymatrix{U\cap V\ar[r]\ar[d] & U\ar[d]^-j \\
V\ar[r]_-i & Q_5}
\end{equation}
in which $U\cap V$ is the affine scheme $Q_3\times \gm$, with coordinates $(x_1,x_2,y_1,y_2,x_3)$. The projections $U\to \A^2\times \gm{}$ and $V\to \A^2\smallsetminus\{0\}\times \A^1$ induced by the projections $Q_5\to \A^3\smallsetminus\{0\}$ are weak equivalences, and it follows that the maps $i$ and $j$ are homotopy trivial. Consequently, the pull-back of the universal rank $2$ bundle on $Q_5$ along both $i$ and $j$ is trivial and we now find explicit trivializations starting with $U$. Consider the matrix
\[
E:=\begin{pmatrix} 0 & 0 & -x_3 \\ 0 & 1 & 0 \\ x_3^{-1} & -x_2x_3^{-1} & x_1\end{pmatrix}\begin{pmatrix} 1 & 0 & 0 \\ y_2 & 1 & 0 \\ -x_3^{-1}y_1 & 0 & 1\end{pmatrix}=\begin{pmatrix} y_1 & 0 & -x_3 \\ y_2 & 1 & 0 \\ y_3 & -x_2x_3^{-1} & x_1 \end{pmatrix}
\]
which is easily seen to be elementary (using Whitehead Lemma). We have 
\[
E^{-1}=\begin{pmatrix} x_1 & x_ 2 & x_3 \\ -x_1y_2 & 1-x_2y_2 & -x_3y_2 \\ -x_3^{-1}(1-x_1y_1) & x_3^{-1}x_2y_1 & y_1\end{pmatrix}
\]
and the properties $(x_1\phantom{i}x_2\phantom{i}x_3)E=(1\phantom{i}0\phantom{i}0)$, $E^{-1}\begin{pmatrix} y_1 \\ y_2 \\y_3\end{pmatrix}=\begin{pmatrix} 1 \\ 0 \\0\end{pmatrix}$ are easily verified. It follows that $E$ induces an isomorphism $k[U]^2\to P_{\vert U}$, for instance since
\[
E^{-1}M(P)E=E^{-1}(\mathrm{Id}-\begin{pmatrix} y_1 \\ y_2 \\y_3\end{pmatrix}(x_1\phantom{i}x_2\phantom{i}x_3))E=\mathrm{Id}-\begin{pmatrix} 1 \\ 0 \\0\end{pmatrix}(1\phantom{i}0\phantom{i}0)=\begin{pmatrix} 0 & 0 & 0 \\ 0 & 1 & 0 \\ 0 & 0 & 1\end{pmatrix}.
\]
We perform the same task over $V$, now using the matrix
\begin{eqnarray*}
F& := & \begin{pmatrix}1 & 0 & (1-x_3)y_1 \\ 0 & 1 & (1-x_3)y_2 \\ 0 & 0 & 1\end{pmatrix}\begin{pmatrix}1 & 0 & 0 \\ 0 & 1 & 0 \\ -x_1 & -x_2 & 1\end{pmatrix}\begin{pmatrix}0 & 0 & -1 \\ 0 & 1 & 0 \\ 1 & 0 & 0\end{pmatrix}\begin{pmatrix} 1 & 0 & 0 \\ y_2 & 1 & 0 \\ -y_1 & 0 & 1\end{pmatrix}\\
& = & \begin{pmatrix} y_1 & -x_2y_1(1-x_3) & -1+(1-x_3)x_1y_1 \\ y_2 & 1-x_2y_2(1-x_3) & (1-x_3)x_1y_2 \\ 0 & -x_2 & x_1\end{pmatrix}
\end{eqnarray*}
which is also elementary. We find
\[
F^{-1}=\begin{pmatrix} x_1 & x_ 2 & x_3 \\ -x_1y_2 & x_1y_1 & -y_2 \\ -x_2y_2 & x_2y_1 & y_1\end{pmatrix}
\]
and $(x_1\phantom{i}x_2\phantom{i}x_3)F=(1\phantom{i}0\phantom{i}0)$, $F^{-1}\begin{pmatrix} y_1 \\ y_2 \\y_3\end{pmatrix}=F^{-1}\begin{pmatrix} y_1 \\ y_2 \\0\end{pmatrix}=\begin{pmatrix} 1 \\ 0 \\0\end{pmatrix}$ in the algebra 
\[
k[V]=k[x_1,x_2,y_1,y_2,x_3]/(x_1y_1+x_2y_2-1).
\]
As above, $F$ induces an isomorphism $k[V]^2\to P_{\vert V}$ since $F^{-1}M(P)F=\begin{pmatrix} 0 & 0 & 0\\ 0 & 1 & 0 \\ 0 & 0 & 1\end{pmatrix}$.

Over the intersection $U\cap V$, $E$ and $F$ do the same job and the above equations show that $E^{-1}F$ induces a matrix in $\mathrm{SL}_2(U\cap V)$. Explicitly, we find
\[
E^{-1}F=\begin{pmatrix} 1 & 0 & 0 \\ 0 &  1-x_2y_2(1-x_3) & x_1y_2(1-x_3) \\
0 & (x_3^{-1}-1)x_2y_1 & x_3^{-1}-x_1y_1(x_3^{-1}-1)\end{pmatrix}
\]
and we set 
\[
T:=\begin{pmatrix} 1-x_2y_2(1-x_3) & x_1y_2(1-x_3) \\ (x_3^{-1}-1)x_2y_1 & x_3^{-1}-x_1y_1(x_3^{-1}-1)\end{pmatrix}\in \mathrm{SL}_2(U\cap V).
\]
A simple computation yields 
\[
T:=\begin{pmatrix} 1 & 0 \\ 0 & x_3^{-1}\end{pmatrix}\begin{pmatrix} x_1 & -y_2 \\ x_2 & y_1\end{pmatrix}\begin{pmatrix} 1 & 0 \\ 0 & x_3^{-1}\end{pmatrix}^{-1}\begin{pmatrix} x_1 & -y_2 \\ x_2 & y_1\end{pmatrix}^{-1},
\]
i.e. $T$ is a commutator.
The base point $1\in\gm{}$ yields an embedding $Q_3\to Q_3\times \gm{}$, while the base point $(1,0,1,0)\in Q_3$ induces an embedding $\gm{}\to Q_3\times \gm{}$. A direct computation shows that the restriction of $T$ along both these morphisms is trivial. Since $\mathrm{SL}_2$ is a sheaf, the morphism $T\colon Q_3\times \gm{}\to \mathrm{SL}_2$ then induces a morphism of sheaves
\[
T\colon Q_3\wedge \gm{}\to \mathrm{SL}_2
\]
that we still denote by $T$, slightly abusing notation.  
Now, the square \eqref{eqn:homotopycartesianQ5} is homotopy cocartesian, being weakly equivalent to the cocartesian square 
\[
\xymatrix{(\A^2\smallsetminus \{0\})\times \gm{}\ar[r]\ar[d] & \A^2\times \gm{}\ar[d] \\
(\A^2\smallsetminus \{0\})\times \A^1 \ar[r] & \A^3\smallsetminus\{0\},}
\] 
and the long exact sequence (of pointed sets and groups) associated to this homotopy cocartesian square induces in particular a map
\[
\xymatrix{[(\A^2\smallsetminus \{0\})\times \gm{},\mathrm{SL}_2]_{\A^1}\ar[r] &[\A^3\smallsetminus\{0\},\mathrm{BSL}_2]_{\A^1}.}
\]
The above discussion shows that the image of $T\in [Q_3\wedge \gm{},\mathrm{SL}_2]_{\A^1}$ under this map is precisely the (homotopy class of) the universal rank $2$ bundle on $Q_5$. In other words, $T$ is the adjoint of the morphism in $[Q_5,\mathrm{BSL}_2]_{\A^1}$ that gives the universal rank $2$ bundle. On the other hand, we have
\[
[\A^3\smallsetminus\{0\},\mathrm{BSL}_2]_{\A^1}=\piaone_{2,3}(\mathrm{BSL}_2)(k)=\KMW {-1}(k)
\]
and the class of the universal bundle is a generator of this group. It follows that $h\cdot T$ is trivial, or in other terms $T$ is $2_{\epsilon}$-torsion. 

\subsubsection{Weight shifting}\label{sec:weightshifting}

As seen in the above section, the morphism $T\colon Q_3\wedge \gm{}\to \mathrm{SL}_2$ is $h=2_{\epsilon}$-torsion. Writing $m\colon Q_3\wedge \gm{}\to Q_3\wedge \gm{}$ the morphism $\mathrm{Id}\wedge (\cdot 2)$, this means by \cite[Proposition 2.1.9]{Asok20b} that the composite 
\[
Q_3\wedge \gm{}\xrightarrow{m}Q_3\wedge \gm{}\xrightarrow{T} \mathrm{SL}_2.
\]
is trivial (up to homotopy). It follows that $T$ is in the kernel ${}_{2_\epsilon}\piaone_{1,3}(\mathrm{SL}_2)$ of the map 
\[
\piaone_{1,3}(\mathrm{SL}_2)=[Q_3\wedge \gm{},\mathrm{SL}_2]_{\A^1}\xrightarrow{2_{\epsilon}} [Q_3\wedge \gm{},\mathrm{SL}_2]_{\A^1}=\piaone_{1,3}(\mathrm{SL}_2)
\] 
induced by $m$. As explained in \cite[Theorem 2.1.11]{Asok20b}, the choice of $\tau=-1$ yields a homomorphism
\[
{}_{2_\epsilon}\piaone_{1,3}(\mathrm{SL}_2)\to \piaone_{2,2}(\mathrm{SL}_2)/[-1]\piaone_{2,3}(\mathrm{SL}_2)
\] 
that is called \emph{weight shifting}. Our purpose in this section is to explicitly perform this procedure that is a priori quite abstract.

For this, it is slightly inconvenient to work with $Q_3\wedge \gm{}$ (which is just a sheaf), and we instead compute the composite
\[
Q_3\times \gm{}\xrightarrow{m'}Q_3\times \gm{}\xrightarrow{T} \mathrm{SL}_2
\]
where $m'$ is a model of $m$. To obtain this map, we consider the morphism $m'\colon Q_5\to Q_5$ defined by
\[
m'(x_i)=\begin{cases} x_i & \text{ for $i=1,2$,} \\ x_3^2 & \text{ for $i=3$,}\end{cases}; \phantom{i} m'(y_i)=\begin{cases} y_i(1+x_3y_3) & \text{ if $i=1,2$.} \\ y_3^2 & \text{ if $i=3$.}\end{cases}
\]
and we observe that it respects the subschemes $U$ and $V$, i.e. it induces endomorphisms of $U$ and $V$, and consequently of $U\cap V$. 
We have 
\[
T\circ m'=\begin{pmatrix} 1-x_2y_2(1-x_3^2) & x_1y_2(1-x_3^2) \\ x_2y_1(x_3^{-2}-1) & x_3^{-2}-x_1y_1(x_3^{-2}-1)\end{pmatrix}
\]
that we may rewrite as follows using the above presentation as a commutator
\[
T\circ m'=\begin{pmatrix} x_3 & 0 \\ 0 & x_3^{-1}\end{pmatrix}\begin{pmatrix} x_1 & -y_2 \\ x_2 & y_1\end{pmatrix}\begin{pmatrix} x_3^{-1} & 0 \\ 0 & x_3\end{pmatrix}\begin{pmatrix} x_1 & -y_2 \\ x_2 & y_1\end{pmatrix}^{-1}.
\]
Since $Q_3\times\gm{}$ is affine and $\mathrm{SL}_2$ is $\A^1$-naive (e.g. \cite{Asok15a}), there should be an explicit naive $\A^1$-homotopy between $T\circ m'$ and the identity. We may use Whitehead lemma to obtain
\[
\begin{pmatrix} x_3 & 0 \\ 0 & x_3^{-1}\end{pmatrix}=\begin{pmatrix} 1 & 0 \\ x_3^{-1} & 1\end{pmatrix}\begin{pmatrix} 1 & 1-x_3  \\ 0 & 1\end{pmatrix}\begin{pmatrix} 1 & 0 \\ -1 & 1\end{pmatrix}\begin{pmatrix} 1 & 1-x_3^{-1}  \\ 0 & 1\end{pmatrix},
\]
from which we deduce the expected homotopy. More precisely, we consider the matrix
\[
A(t):=\begin{pmatrix} 1 & 0 \\ tx_3^{-1} & 1\end{pmatrix}\begin{pmatrix} 1 & t(1-x_3)  \\ 0 & 1\end{pmatrix}\begin{pmatrix} 1 & 0 \\ -t & 1\end{pmatrix}\begin{pmatrix} 1 & t(1-x_3^{-1})  \\ 0 & 1\end{pmatrix}, 
\]
and we obtain a matrix
\[
D(t):=A(t)\begin{pmatrix} x_1 & -y_2 \\ x_2 & y_1\end{pmatrix}A(t)^{-1}\begin{pmatrix} x_1 & -y_2 \\ x_2 & y_1\end{pmatrix}^{-1}
\]
which has the property that $D(0)=\mathrm{Id}$ and $D(1)=T\circ m'$. Its restriction to both $Q_3\times \{1\}\times \A^1$ and $\{1,0,1,0\}\times \gm{}\times \A^1$ is trivial, and consequently $D$ can be seen as a map 
\[
(Q_3\wedge \gm{})\times \A^1\to \mathrm{SL}_2.
\] 
Setting $x_3=-1$, we obtain a map $Q_3\to Q_3\wedge \gm{}$, which in turn yields a map $G(t)\colon Q_3\times \A^1\to \mathrm{SL}_2$, which becomes trivial (since $D(1)=T\circ m'$ and $m'(-1)=1$) when restricted to $Q_3\times S^0$, where we set $S^0:=\{t(1-t)=0\}\subset \A^1$. Thus, we obtain a map $(Q_3\times\A^1)/(Q_3\times S^0)\to \mathrm{SL}_2$ that is explicitly of the form
\[
G(t)  =  \begin{pmatrix} 1-2t^2 & 4(t-t^3) \\ -2(t-t^3) & 1-6t^2+4t^4 \end{pmatrix}\begin{pmatrix} x_1 & -y_2 \\ x_2 & y_1\end{pmatrix}\begin{pmatrix} 1-6t^2+4t^4 & 4(t^3-t) \\ 2(t-t^3) & 1-2t^2 \end{pmatrix}\begin{pmatrix} x_1 & -y_2 \\ x_2 & y_1 \end{pmatrix}^{-1}. 
\]
In more traditional terms, $G(t)\in \mathrm{SL}_2(k[Q_3][t],t(1-t))$. Now, remember that we have a morphism
\[
\alpha_3\colon Q_3\times \A^1\to Q_4
\]
given by $\alpha_3(x_1,x_2,y_1,y_2,t)=(x_1,x_2,y_1t(1-t),y_2t(1-t),t)$, mapping $ Q_3\times S^0$ to the closed subset  $\{z(1-z)=0\}\subset Q_4$. On the other hand, the above morphism induces an isomorphism 
\[
\alpha_3\colon Q_3\times (\A^1\smallsetminus \{0,1\})\simeq D(z(1-z))\subset Q_4
\]
with inverse given by $(x_1,x_2,y_1,y_2,z)\mapsto (x_1,x_2,\frac{y_1}{z(1-z)},\frac{y_2}{z(1-z)},z)$. The matrix $G(t)$ restricts to a map
\[
Q_3\times (\A^1\smallsetminus \{0,1\})\to \mathrm{SL}_2
\]
and thus induces a map $D(z(1-z))\to \mathrm{SL}_2$ that should morally be extended to $Q_4$ by setting its restriction to $\{z(1-z)=0\}$ to be trivial. However, $G(t)$ lifts to a map on $Q_4$ only if every monomial involving $\frac{y_i}{z(1-z)}$ is multiplied by $z(1-z)$, which is not the case as a simple computation shows. We can nevertheless modify $G(t)$ up to homotopy relative to $t(1-t)$ to achieve this, using the next lemma.  

\begin{lem}
Let $M:=\begin{pmatrix} x_1 & -y_2 \\ x_2 & y_1\end{pmatrix}$. Then:
\[
M\begin{pmatrix} s & p \\ q & r\end{pmatrix}M^{-1}=s\begin{pmatrix} x_1 \\ x_2 \end{pmatrix}\begin{pmatrix} y_1 & y_2\end{pmatrix}+ p\begin{pmatrix} x_1 \\ x_2 \end{pmatrix}\begin{pmatrix} -x_2 & x_1\end{pmatrix} + q\begin{pmatrix} -y_2 \\ y_1 \end{pmatrix}\begin{pmatrix} y_1 & y_2\end{pmatrix} + r\begin{pmatrix} -y_2 \\ y_1 \end{pmatrix}\begin{pmatrix} -x_2 & x_1\end{pmatrix}
\] 
for any $s,p,q,r\in k[t]$.
\end{lem}

Now, 
\[
G(t)=\begin{pmatrix} 1-2t^2 & 4(t-t^3) \\ -2(t-t^3) & 1-6t^2+4t^4 \end{pmatrix}\begin{pmatrix} x_1 & -y_2 \\ x_2 & y_1\end{pmatrix}\begin{pmatrix} 1-6t^2+4t^4 & 4(t^3-t) \\ 2(t-t^3) & 1-2t^2 \end{pmatrix}\begin{pmatrix} x_1 & -y_2 \\ x_2 & y_1 \end{pmatrix}^{-1}.
\]
Using this lemma, we replace $\begin{pmatrix} 1-6t^2+4t^4 & 4(t^3-t) \\ 2(t-t^3) & 1-2t^2 \end{pmatrix}$ by a matrix $B(t) = \begin{pmatrix} s & p \\ q & r\end{pmatrix}$ for some polynomials $s,p,q,r \in k[t]$ assumed to be invertible and homotopic to the latter relative to $t(1-t)$, and having the properties that the product 
\[
\begin{pmatrix} x_1 & -y_2 \\ x_2 & y_1\end{pmatrix}B(t)\begin{pmatrix} x_1 & -y_2 \\ x_2 & y_1\end{pmatrix}^{-1}
\] 
lies in the image of the $k$-algebra map $\psi \colon k[Q_4] \to k[Q_3\times \A^1]$ induced by $\alpha_3$, i.e.  
\[
x_i \mapsto x_i, y_i \mapsto t(1-t)y_i, z \mapsto t.
\]
Let then 
\[ 
B(t) = \begin{pmatrix} s & p \\ q & r\end{pmatrix} = \left (
\begin{pmatrix} 1 & \alpha t(1-t)-2t \\ 0 & 1\end{pmatrix}\begin{pmatrix} 1 & 0 \\ t+\beta t(1-t) & 1\end{pmatrix}\right)^2
\] 
for some polynomials $\alpha,\beta\in k[t]$. A direct computation yields 
\[
q= 2(t+t(1-t)\beta)+(t+t(1-t)\beta)^2(t(1-t)\alpha-2t).
\] 
If we want the expression to be divisible by $t^2(1-t)^2$, we may for instance take $\alpha=-6$, $\beta=-1$ to obtain a matrix
\[
B(t) := \left(\begin{pmatrix}
    1 & -2t - 6t(1-t) \\ 0 & 1 
\end{pmatrix}\begin{pmatrix}
    1 & 0 \\ t - t(1-t) & 1
\end{pmatrix}\right)^2 = \left(\begin{pmatrix}
    1 & 6t^2-8t \\ 0 & 1 
\end{pmatrix}\begin{pmatrix}
    1 & 0 \\ t^2 & 1
\end{pmatrix}\right)^2. 
\]
that is apparently homotopic to $A(t)^{-1}$ relative to $t(1-t)$. We can then set 
\[
G'(t) = \begin{pmatrix} 1-2t^2 & 4(t-t^3) \\ -2(t-t^3) & 1-6t^2+4t^4 \end{pmatrix}\begin{pmatrix} x_1 & -y_2 \\ x_2 & y_1\end{pmatrix}B(t)\begin{pmatrix} x_1 & -y_2 \\ x_2 & y_1\end{pmatrix}^{-1}
\] 
and observe that this matrix indeed extends to a map $Q_4\to  \mathrm{SL}_2$ as its coordinates lie in the image of the $k$-algebra map $\psi \colon k[Q_4] \to k[Q_{3}\times \A^1]$. By computing $\psi^{-1}(G'(t))$ we get a map $Q_4 \to Q_3$ given by the unimodular row $(a,b)$, where
\begin{align*}
   a =& 144x_2^2z^9+144x_2y_1z^9+72x_1x_2z^8-384x_2^2z^8-240x_2y_1z^8-72x_2 y_2z^8-72z^{10}-192x_1x_2z^7\\ &+112x_2^2z^7-128x_2y_1z^7+120x_2y_2z^7+192 z^9+92x_1x_2z^6+384x_2^2z^6+256x_2y_1z^6+28x_2y_2z^6\\&-92z^8+96x_1x_2z ^5-208x_2^2z^5+48x_2y_1z^5-24y_1^2z^5-68x_2y_2z^5-96z^7-40x_1x_2z^4- 64x_2^2z^4\\&-16x_2y_1z^4-16y_1^2z^4-28x_2y_2z^4+12y_1y_2z^4+28z^6-32 x_1x_2z^3-48x_2^2z^3-64x_2y_1z^3+16y_1^2z^3\\&+4x_2y_2z^3+8y_1y_2z^3+48 z^5-12x_1x_2z^2+64x_2^2z^2+16y_1^2z^2+16x_2y_2z^2-2y_1y_2z^2+18z^4\\&+ 16x_1x_2z+8y_1^2z-4y_1y_2z-24z^3-2y_1y_2-2z^2+1, \\
   b = & -144x_1x_2z^9+144x_2y_2z^9-72x_1^2z^8+384x_1x_2z^8+72x_1y_2z^8-240x_2 y_2z^8+192x_1^2z^7- 112x_1x_2z^7\\&-120x_1y_2z^7-128x_2y_2z^7-92x_1^2z^6 -384x_1x_2z^6-28x_1y_2z^6+256x_2y_2z^6-96x_1^2z^5+208x_1x_2z^5\\&+68x_1 y_2z^5+48x_2y_2z^5-24y_1y_2z^5-24z^7+40x_1^2z^4+64x_1x_2z^4+28x_1y_2 z^4-16x_2y_2z^4\\&-16y_1y_2z^4+12y_2^2z^4+32z^6+32x_1^2z^3+48x_1x_2z^3- 4x_1y_2z^3-64x_2y_2z^3+16y_1y_2z^3+8y_2^2z^3\\&+24z^5+12x_1^2z^2-64x_1 x_2z^2-16x_1y_2z^2+16y_1y_2z^2-2y_2^2z^2-32z^4-16x_1^2z+8y_1y_2z\\&-4 y_2^2z-4z^3-2y_2^2+4z.
\end{align*}

\begin{rem}
We note that both the exotic Hopf maps obtained in Section \ref{sec:symplecticHopfmap} and above are defined over $\Z$. However, they reduce to a homotopically trivial map over $\mathbb{F}_2$ as a direct check shows. In the second case, this can be explained by the fact that we had to choose a primitive $2$nd root of $1$. In the first case, we used a computation of $\mathrm{GW}_1^0(k)$ that needs $2$ to be invertible, and the generator we used is trivial in $\mathrm{GW}_1^0(\mathbb{F}_2)$.
\end{rem}

\begin{rem}
A direct computation shows that the above map is not trivial relative $z(1-z)$. This could probably be further achieved by modifying it up to homotopy, but we haven't carried out the necessary computations since we don't need this property in the sequel.
\end{rem}

\begin{rem}
It would probably be difficult to prove that the exotic Hopf maps of Section \ref{sec:symplecticHopfmap} and the one above are homotopic. They however coincide after base change to $\cplx$, since $[Q_4,Q_3]_{\A^1}=\piaone_{2,2}(Q_3)(\cplx)\simeq \Z/2$ and both maps are non-trivial (their complex realizations being non-trivial).
\end{rem}


\section{A rank $2$ bundle on $\Proj^3_k$}

We finally put everything together to construct the rank $2$ bundle on $\mathbb{J}^3$. Indeed, we consider the following composite
\[
\mathbb{J}^3\xrightarrow{\pi_3} Q_6\xrightarrow{\Sigma_{\pone}\eta'}Q_5\to \mathrm{BSL}_2
\]
where $\pi_3$ is the morphism described in Section \ref{sec:collapsemap}, and $\Sigma_{\pone}\eta'$ is the $\pone$-suspension of either the exotic Hopf map of Section \ref{sec:weightshifting} or Section \ref{sec:symplecticHopfmap}. Explicitly, it is given by pulling back the idempotent matrix \eqref{eqn:universalidempotent} 
\[
M(P)=\begin{pmatrix} 1-x_1y_1 & -x_2y_1 & -x_3y_1 \\ -x_1y_2 & 1-x_2y_2 & -x_3y_2 \\ -x_1y_3 & -x_2y_3 & 1-x_3y_3\end{pmatrix}
\]
on $Q_5$ along $\mathbb{J}^3\xrightarrow{\pi_3} Q_6\xrightarrow{\Sigma_{\pone}\eta'}Q_5$.

\subsection{Invariants of the bundle}
We may compute the invariants associated to the above rank $2$ bundle by looking at its class in the set $[\mathbb{J}^3, \mathrm{BSL}_2]_{\A^1}=[\mathbb{P}^3, \mathrm{BSL}_2]_{\A^1}$.
Recall from \cite[\S 6]{Asok12a} that if $X$ is a smooth $k$-scheme of dimension $\leq 3$, there is an exact sequence of pointed sets
\[
\mathrm{H}^1(X,\piaone_2(\mathrm{BSL}_2))\to \mathrm{H}^3(X,\piaone_3(\mathrm{BSL}_2))\to [X,\mathrm{BSL}_2]_{\A^1}\to \mathrm{H}^2(X,\piaone_2(\mathrm{BSL}_2))\to *
\]
deduced from the Postnikov tower of $\mathrm{BSL}_2$ (using the fact that $\mathrm{BSL}_2$ is $\A^1$-simply connected). We know that $\piaone_2(\mathrm{BSL}_2)=\KMW{2}$, while the sheaf $\piaone_3(\mathrm{BSL}_2)$ fits into the exact sequence \eqref{eqn:pi2a2minus0}
\[
\KMW4/12_\epsilon\to \piaone_3(\mathrm{BSL}_2)\to \mathbf{GW}_3^2\to 0.
\]
We suppose that the base field $k$ is algebraically closed, in which case we obtain from \cite[Theorem 11.7]{Fasel09d}
\[
\mathrm{H}^3(\mathbb{P}^3,\KMW4/12_\epsilon)=\mathrm{H}^3(\mathbb{P}^3,\mathbf{K}^{\mathrm{M}}_4/12)=k^\times/(k^\times)^{12}=0,
\] 
and the above exact sequence then reads 
\[
\mathrm{H}^1(\mathbb{P}^3,\KMW2)\to \mathrm{H}^3(\mathbb{P}^3,\mathbf{GW}_3^2)\to [\mathbb{P}^3,\mathrm{BSL}_2]_{\A^1}\to \mathrm{H}^2(\mathbb{P}^3,\KMW2)\to *
\]
The inclusion $i\colon\mathbb{P}^3_k\smallsetminus \{0\}\subset \mathbb{P}^3_k$ (for which an affine model is given by the inclusion $\mathbb{J}^3\smallsetminus\{\A^3\}\to \mathbb{J}^3$) yields a commutative diagram
\[
\xymatrix@C=1em{\mathrm{H}^1(\mathbb{P}^3_k,\KMW2)\ar[r]\ar[d]^-{i^*} & \mathrm{H}^3(\mathbb{P}^3_k,\mathbf{GW}_3^2)\ar[r]\ar[d]^-{i^*} & [\mathbb{P}^3_k,\mathrm{BSL}_2]_{\A^1}\ar[r]\ar[d]^-{i^*} & \mathrm{H}^2(\mathbb{P}^3_k,\KMW2)\ar[r]\ar[d]^-{i^*} & \star \\
\mathrm{H}^1(\mathbb{P}^3_k\smallsetminus \{0\},\KMW2)\ar[r] & \mathrm{H}^3(\mathbb{P}^3_k\smallsetminus \{0\},\mathbf{GW}_3^2)\ar[r] & [\mathbb{P}^3_k\smallsetminus \{0\},\mathrm{BSL}_2]_{\A^1}\ar[r] & \mathrm{H}^2(\mathbb{P}^3_k\smallsetminus \{0\},\KMW2)\ar[r] & \star }
\]
The projection $\mathbb{P}^3\smallsetminus\{0\}\to \mathbb{P}^2$ defined by $[x_0:x_1:x_2:x_3]\mapsto [x_0:x_1:x_2]$ is a weak equivalence and consequently $\mathbb{P}^3\smallsetminus\{0\}$ is of $\A^1$-cohomological dimension $\leq 2$. The bottom sequence then reduces to a bijection
\[
[\mathbb{P}^3_k\smallsetminus \{0\},\mathrm{BSL}_2]_{\A^1}\to \mathrm{H}^2(\mathbb{P}^3_k\smallsetminus \{0\},\KMW2)
\] 
given by the Euler class. It follows from \cite[Theorem 11.7]{Fasel09d} that the projection map
\[
\mathrm{H}^2(\mathbb{P}^3,\KMW2)\to \mathrm{H}^2(\mathbb{P}^3,\mathbf{K}^{\mathrm{M}}_2)=\mathrm{CH^2}(\mathbb{P}^3)=\Z
\]
is an isomorphism. The same applies for $\mathbb{P}^2_k$ and thus the above diagram is of the form
\[
\xymatrix@C=1.5em{\mathrm{H}^1(\mathbb{P}^3_k,\KMW2)\ar[r] & \mathrm{H}^3(\mathbb{P}^3_k,\mathbf{GW}_3^2)\ar[r]\ar[d] & [\mathbb{P}^3_k,\mathrm{BSL}_2]_{\A^1}\ar[r]^-{c_2}\ar[d]^-{i^*} & \Z\ar[r]\ar@{=}[d] & \star \\
 & \star\ar[r] & [\mathbb{P}^2_k,\mathrm{BSL}_2]_{\A^1}\ar[r]_-{c_2} & \Z\ar[r] & \star }
\]
By construction, the bundle constructed using the exotic Hopf map is classified by the composite 
\[
\mathbb{J}^3\to Q_6\xrightarrow{\Sigma_{\pone}\eta'}Q_5\to \mathrm{BSL}_2
\]
The sequence $\mathbb{J}^3\smallsetminus \A^3\xrightarrow{i} \mathbb{J}^3\to Q_6$ being isomorphic to the cofiber sequence 
\[
\mathbb{P}^3_k\smallsetminus\{0\}\to \mathbb{P}^3_k\to (\pone)^{\wedge 3}
\]
we obtain that the composite $\mathbb{J}^3\smallsetminus \A^3\xrightarrow{i} \mathbb{J}^3\to Q_6$ is trivial up to homotopy. The diagram then shows in particular that the second Chern class of this bundle vanishes. We can now use the morphism $\mathbb{P}^3_k\to (\pone)^{\wedge 3}$, modeled by the morphism $\mathbb{J}^3\to Q_6$, to obtain a commutative diagram with exact columns
\[
\xymatrix@C=1.5em{\mathrm{H}^1((\pone)^{\wedge 3},\KMW2)\ar[r] & \mathrm{H}^3((\pone)^{\wedge 3},\mathbf{GW}_3^2)\ar[r]\ar[d] & [(\pone)^{\wedge 3},\mathrm{BSL}_2]_{\A^1}\ar[r]\ar[d]^-{i^*} & \mathrm{H}^2((\pone)^{\wedge 3},\KMW2)\ar[r]\ar@{=}[d] & \star \\
\mathrm{H}^1(\mathbb{P}^3_k,\KMW2)\ar[r] & \mathrm{H}^3(\mathbb{P}^3_k,\mathbf{GW}_3^2)\ar[r]\ar[d] & [\mathbb{P}^3_k,\mathrm{BSL}_2]_{\A^1}\ar[r]^-{c_2}\ar[d]^-{i^*} & \Z\ar[r]\ar@{=}[d] & \star \\
 & \star\ar[r] & [\mathbb{P}^2_k,\mathrm{BSL}_2]_{\A^1}\ar[r]_-{c_2} & \Z\ar[r] & \star }
\]
In view of \cite[Proposition 1.1.5]{Asok16}, the only nontrivial group in the top row is 
\[
\mathrm{H}^3((\pone)^{\wedge 3},\mathbf{GW}_3^2)=\mathbf{GW}_0^3(k)=\Z/2\Z,
\] and we thus obtain 
\[
\xymatrix@C=1.5em{\star\ar[r] & \mathbf{GW}_0^3(k)\ar[r]\ar[d] & [(\pone)^{\wedge 3},\mathrm{BSL}_2]_{\A^1}\ar[r]\ar[d]^-{i^*} & \star\ar[r]\ar@{=}[d] & \star \\
\mathrm{H}^1(\mathbb{P}^3_k,\KMW2)\ar[r] & \mathrm{H}^3(\mathbb{P}^3_k,\mathbf{GW}_3^2)\ar[r]\ar[d] & [\mathbb{P}^3_k,\mathrm{BSL}_2]_{\A^1}\ar[r]^-{c_2}\ar[d]^-{i^*} & \Z\ar[r]\ar@{=}[d] & \star \\
 & \star\ar[r] & [\mathbb{P}^2_k,\mathrm{BSL}_2]_{\A^1}\ar[r]_-{c_2} & \Z\ar[r] & \star }
\]
We claim that he composite
\[
\Z/2=\mathbf{GW}_0^3(k)=\mathrm{H}^3((\pone)^{\wedge 3},\mathbf{GW}_3^2)\to \mathrm{H}^3(\mathbb{P}^3_k,\mathbf{GW}_3^2)\to [\mathbb{P}^3_k,\mathrm{BSL}_2]_{\A^1}
\]
is split injective.

Indeed, we consider the map $[\mathbb{P}^3_k,\mathrm{BSL}_2]_{\A^1}\to [\mathbb{P}^3_k,\mathrm{BSp}]_{\A^1}=\widetilde{\mathrm{GW}}^2(\mathbb{P}_k^3)$ induced by the stabilization, where the right-hand side is the reduced Grothendieck-Witt group of symplectic bundles. The push-forward map induced by the projection $\mathbb{P}^3_k\to \Spec(k)$ reads as
\[
\widetilde{\mathrm{GW}}^2(\mathbb{P}_k^3)\to \mathrm{GW}^3(\Spec(k))=\Z/2
\]
and is explicitly described in \cite[\S 4.2.2]{Asok25a}, where it is called the motivic Atiyah-Rees invariant. The computation therein shows that this homomorphism provides a splitting of $\Z/2\to [\mathbb{P}^3_k,\mathrm{BSL}_2]_{\A^1}$, and thus that the Atiyah-Rees invariant of the exotic Hopf bundle is $1\in \Z/2$.


\begin{footnotesize}
\bibliographystyle{alpha}
\bibliography{General}
\end{footnotesize}
\Addresses
\end{document}